\documentclass{groups17}
\usepackage{amscd}
\usepackage{graphicx}
\usepackage[normalem]{ulem}
\usepackage{pstricks}
\usepackage{calc}
\usepackage{enumitem}

\newcommand{\pres}[2]{\langle {#1}\ |\ {#2} \rangle}
\newcommand{\gpres}[1]{\langle {#1} \rangle}
\newcommand{\ngpres}[1]{\langle\langle {#1} \rangle\rangle}
\newcommand{\sgp}[1]{\langle {#1}\rangle}
\newcommand{\gwell}{l}
\usepackage{amsmath}
\makeatletter
\newcommand{\eqnum}{\leavevmode\hfill\refstepcounter{equation}\textup{\tagform@{\theequation}}}

\def\Z{\mathbb Z}
\def\P{\mathcal{P}}
\def\im{{\rm im \,}}
\def\ra{\rightarrow}

\def\<{\langle}
\def\>{\rangle}

\newtheorem{conj}[thm]{Conjecture}

\begin{document}
\maketitle{ASPHERICAL RELATIVE PRESENTATIONS ALL OVER AGAIN}{%
  WILLIAM A. BOGLEY$^{\ast}$ and MARTIN EDJVET$^{\dagger}$ and GERALD WILLIAMS$^{\S}$}{
\newline  $^{\ast}$Department of Mathematics, Oregon State University, Kidder Hall 368, Corvallis, OR 97331-4605, USA\newline
  Email: Bill.Bogley@oregonstate.edu\\[2pt]
  $^{\dagger}$School of Mathematical Sciences, University of Nottingham, University Park, Nottingham, NG7 2RD, U.K.\newline
  Email: martin.edjvet@nottingham.ac.uk\\[2pt]
  $^{\S}$Department of Mathematical Sciences, University of Essex, Wivenhoe Park, Colchester,
Essex CO4 3SQ, U.K.\newline
  Email: gerald.williams@essex.ac.uk
  }

\begin{center}
\em  In gratitude to Stephen J.~Pride for his friendship and mentorship over the years.
\end{center}

\begin{abstract}
The concept of asphericity for relative group presentations was introduced twenty five years ago. Since then, the subject has advanced and detailed asphericity classifications have been obtained for various families of one-relator relative presentations. Through this work the definition of asphericity has evolved and new applications have emerged.

In this article we bring together key results on relative asphericity, update them, and exhibit them under a single set of definitions and terminology. We describe consequences of asphericity and present techniques for proving asphericity and for proving non-asphericity. We give a detailed survey of results concerning one-relator relative presentations where the relator has free product length four.
\end{abstract}

\section{Introduction}

Consider a \em relative group presentation \em $\mathcal{P} = \pres{G,\mathbf{x}}{\mathbf{r}}$. Thus, $G$ is a group, $\mathbf{x}$ is a set disjoint from $G$, and $\mathbf{r}$ is a set of words in the free product $G \ast F$, where $F$ denotes the free group with basis $\mathbf{x}$. The group defined by $\mathcal{P}$ is the quotient $G(\mathcal{P}) = (G \ast F)/R$ where $R$ is the normal closure of $\mathbf{r}$ in $G \ast F$.

Relative presentations and the groups they define have long been studied in connection with equations over groups (e.g.\ \cite{Freedman78}, \cite[Theorem~3]{GR}, \cite[Section~6.2]{MKS}, \cite{NeumannRoot}), where the objective is to determine conditions on $\mathcal{P}$ under which the natural homomorphism $G \ra G(\mathcal{P})$ is injective, so that $G$ can be viewed as a subgroup of $G(\mathcal{P})$. When the relative presentation $\mathcal{P}$ is aspherical, in a suitable sense, the cohomology and finite subgroups of $G(\mathcal{P})$ can be completely described in terms of those of $G$. In Section~\ref{section:Concepts} we describe asphericity concepts in algebraic topological terms. Methods from combinatorial geometry for proving asphericity are the focus in Section~\ref{section:Methods}. In Section~\ref{section:LengthFour} we survey asphericity classifications that have been obtained for one-relator relative presentations where the relator has free product length four.

\section{Asphericity concepts}\label{section:Concepts}

We refer the reader to \cite{BrownBook},\cite{BrownCohNote},\cite{HatcherAT},\cite{AJSAlgTop} for general aspects of algebraic topology and the cohomology of groups. A topological space $X$ is \em aspherical \em if each spherical map $S^k \ra X$, $k \geq 2$, is homotopic to a constant map. A connected CW complex $K$ is aspherical if and only if its universal covering complex $\widetilde{K}$ is acyclic, in which case $\widetilde{K}$ is actually contractible. The homotopy type of an aspherical CW complex $K$ is uniquely determined by its fundamental group and if $\pi_1K \cong G$ then $K$ is called an \em Eilenberg-Mac~Lane complex \em of type $K(G,1)$ and all homotopy invariants of the complex $K$ are algebraic invariants of the group $G$. In particular, the augmented cellular chain complex $\mathcal{C}_\ast(\widetilde{K}) \ra \Z$ provides a resolution of the trivial module $\Z$ by free $\Z G$-modules, which in turn determines the homology and cohomology groups of $G$.

Asphericity concepts for ordinary group presentations are well established \cite{CCH}, \cite{LyndonSchupp}, \cite{Ol91}. For example, the Lyndon Identity Theorem~\cite{LynId} implies that each one-relator presentation for a group $G$ is \em combinatorially aspherical, \em which means that a $K(G,1)$ can be constructed as in \cite{DV73}.

\subsection{Asphericity in terms of the cellular model}

In his work on equations over locally indicable groups, Howie \cite{How81} employed topological methods involving a relative two-complex $(L,K)$ associated to a relative presentation $\mathcal{P} = \pres{G,\mathbf{x}}{\mathbf{r}}$. Here, we take $K$ to be a $K(G,1)$-complex. The complex $L$ is then obtained by attaching two-cells to the one-point union $K \vee \bigvee_\mathbf{x} S^1_x$ via cellular loops $\dot{\varphi}_r:S^1 \ra K \vee \bigvee_\mathbf{x} S^1_x$ that realize the homotopy classes $[\dot{\varphi}_r] \equiv r \in G \ast F \cong \pi_1(K \vee \bigvee_\mathbf{x} S^1_x)$:
\begin{equation*}\label{eqn:Model}
L = K \vee \bigvee_\mathbf{x} S^1_x \cup \bigcup_\mathbf{r} c^2_r.
\end{equation*}
The relative two-complex $(L,K(G,1))$ is the \em cellular model \em  of $\mathcal{P}$. Because $K = K(G,1)$ is chosen to be aspherical, the homotopy types of $L$ and of the pair $(L,K)$ are both uniquely determined by the relative presentation $\mathcal{P}$. We denote $L$ by $L(\P)$.

We are interested in the situation where $L(\P)$ is aspherical and so is a $K(G(\P),1)$. The following technical lemma will be used repeatedly.

\begin{lemma}\label{Lemma:Covering} Suppose that $(Y,X)$ is a relative two-complex where $X$ and $Y$ are connected CW complexes and the inclusion-induced homomorphism $\pi_1X \ra \pi_1Y$ is injective. If $\pi_2Y = 0$, then $Y$ is aspherical if and only if $X$ is aspherical.
\end{lemma}

\begin{proof}
Let $p:\widetilde{Y} \ra Y$ be the universal covering projection. Since $\pi_1 X \ra \pi_1Y$ is injective, each component of the pre-image $p^{-1}(X)$ is a copy of the universal covering complex $\widetilde{X}$ and so $H_k(p^{-1}(X))$ is a direct sum of copies of $H_k \widetilde{X}$. Since $\pi_2Y \cong \pi_2\widetilde{Y} \cong H_2\widetilde{Y}$ and the complement $Y-X$ has no cells in dimensions three and up, the result follows by considering the long exact homology sequence for the pair $(\widetilde{Y},p^{-1}(X))$.
\end{proof}

The term `aspherical' has been applied to relative group presentations in a variety of ways. The most general of these is given in Definition~\ref{Definition:Aspherical} below. This concept appeared implicitly in \cite[Theorem~4.1]{BP} and explicitly as a defined concept in \cite{BW1}; it is expressed in terms of the second relative homotopy group $\pi_2(L(\P),K(G,1))$ of the cellular model. This formulation is broad enough to accommodate previous incarnations of the term `aspherical', but is also specific enough to activate the principal group-theoretic and topological consequences of asphericity such as those described in Theorem~\ref{Theorem:Theory} and Theorem~\ref{thm:3mfdcriterion}, below. In addition to the intrinsic value of a flexible definition of a basic concept, this definition of asphericity clearly delineates the concept of non-asphericity in an algebraic way. This is reflected in recent and ongoing work, such as \cite{BShift},\cite{BW1},\cite{BW2},\cite{HW16},\cite{McD17}, in which, when studying certain families of relative presentations, interesting groups and presentations are uncovered in the non-aspherical realm.

\begin{de}[See \textbf{\cite[Section~4]{BP},\cite{BW1}}]\label{Definition:Aspherical}A relative presentation $\mathcal{P} = \pres{G,\mathbf{x}}{\mathbf{r}}$ is \emph{aspherical} if the second relative homotopy group $\pi_2(L(\P),K(G,1))$ is trivial.
\end{de}

\begin{lemma}\label{Lemma:KG1} A relative presentation $\mathcal{P} = \pres{G,\mathbf{x}}{\mathbf{r}}$ is aspherical if and only if the natural homomorphism $G \ra G(\mathcal{P})$ is injective and $L(\P)$ is a $K(G(\mathcal{P}),1)$-complex.
\end{lemma}

\begin{proof}
Letting $(L,K) = (L(\P),K(G,1))$, the natural homomorphism $G \ra G(\mathcal{P})$ is identified with the inclusion-induced homomorphism $\pi_1K \ra \pi_1L$, which in turn appears in the long exact sequence for the pair $(L,K)$:
\begin{equation*}\label{eqn:LEHS}
0 = \pi_2K \ra \pi_2 L \ra \pi_2(L,K) \ra \pi_1K \ra \pi_1L \ra 1.
\end{equation*}
This shows that if $\pi_2(L,K) = 0$ then $G \ra G(\mathcal{P})$ is injective and $\pi_2L = 0$, so that $L$ is aspherical by Lemma~\ref{Lemma:Covering}. The converse statement follows immediately from the long exact sequence. \end{proof}

If $\P = \pres{G,\mathbf{x}}{\mathbf{r}}$ is aspherical, then not only is the natural homomorphism $G \ra G(\mathcal{P})$ injective, so that we can view $G$ as a subgroup of $G(\P)$, but $G$ plays a determining role in the subgroup and homological structure of $G(\mathcal{P})$. More precisely we have the following.

\begin{thm}\label{Theorem:Theory}\emph{(\textbf{Compare \cite{BP}})} Let $\mathcal{P} = \pres{G,\bf{x}}{\bf r}$ be an aspherical relative presentation.
\begin{enumerate}
\item[(a)] If $\Gamma$ is a subgroup of $G(\P)$ such that $\Gamma \cap G = 1$, then $\Gamma$ has geometric dimension at most two.
\item[(b)] For $n \geq 3$, the embedding $G \ra G(\mathcal{P})$ determines natural equivalences of functors defined on the category of $\Z G(\mathcal{P})$-modules:
\begin{eqnarray*}
H_n(G,\mathrm{Res}^{G(\mathcal{P})}_G(-)) &\ra & H_n(G(\mathcal{P}),-),\\
H^n(G(\mathcal{P}),-) &\ra& H^n(G,\mathrm{Res}^{G(\mathcal{P})}_G(-)).
\end{eqnarray*}
\item[(c)] (\emph{\textbf{\cite[Section~3]{Hueb}, attributed to Serre}}; see also \emph{\textbf{\cite{FR}}} and \emph{\textbf{\cite[Theorem~5]{HS}}}) Every finite subgroup of $G(\mathcal{P})$ is conjugate to a subgroup of $G$. Moveover if $w \in G(\mathcal{P})$ and the intersection $G \cap wGw^{-1}$ contains a non-trivial element of finite order, then $w \in G$.
\end{enumerate}
\end{thm}

\begin{proof} Let $(L,K) = (L(\P),K(G,1))$. Given $\Gamma \leq G(\P)$ as in (a), there is a covering projection $p:\overline{L} \ra L$ for which $\pi_1\overline{L} \cong \Gamma$. Since $\Gamma \cap G = 1$, the pre-image $p^{-1}(K)$ is a disjoint union of copies of the contractible universal covering complex $\widetilde{K(G,1)}$. Since the inclusion of $p^{-1}(K)$ in $\overline{L}$ is a cofibration, the quotient map $q:\overline{L} \ra X$ that collapses each component of $p^{-1}(K)$ to a zero-cell is a homotopy equivalence, so $X$ is an aspherical two-complex with fundamental group $\Gamma$. To prove (b), denote the simply connected covering complexes by $\widetilde{L}$ and $\widetilde{K}$. Since $G$ embeds in $G(\P)$ by Lemma~\ref{Lemma:KG1}, there is a chain isomorphism $\mathcal{C}_\ast(\widetilde{L}) \cong \Z G(\mathcal{P}) \otimes_G \mathcal{C}_\ast(\widetilde{K})$ in dimensions $\ast \geq 3$, so if $M$ is an arbitrary (right) $\Z G(\mathcal{P})$-module and $n \geq 3$, then\footnote{For the isomorphisms when $n=3$, the fact that each three-cell of $L$ is contained in $K$ implies that $\ker(M \otimes_{G(\P)} C_3 \widetilde{L} \ra M \otimes_{G(\P)} C_2 \widetilde{L} ) = \ker(M \otimes_{G} C_3\widetilde{K} \ra M \otimes_{G} C_2\widetilde{K})$ and $\im(\mathrm{Hom}_{G(\P)}(C_2\widetilde{L}, M) \ra \mathrm{Hom}_{G(\P)}(C_3\widetilde{L}, M)) = \im(\mathrm{Hom}_{G}(C_2\widetilde{K}, M) \ra \mathrm{Hom}_{G}(C_3\widetilde{K}, M)).$}
\begin{eqnarray*}
H_n(G(\mathcal{P}),M) &=& H_n(M \otimes_{G(\mathcal{P})} \mathcal{C}_\ast(\widetilde{L}))\\
&\cong& H_n(M \otimes_{G(\mathcal{P})} \Z G(\mathcal{P}) \otimes_G \mathcal{C}_\ast(\widetilde{K}))\\
&\cong& H_n(M \otimes_G \mathcal{C}_\ast(\widetilde{K}))\\
&\cong& H_n(G,\mathrm{Res}^{G(\mathcal{P})}_G(M)).
\end{eqnarray*}
Similar calculations involving $\mathrm{Hom}$ lead to cohomology isomorphisms as in (b). The statement (c) is a direct consequence of (b) as discussed in the cited references.
\end{proof}

\begin{re}\label{rem:geomdim}
If $G $ is trivial and $K = K(G,1)$ is a point, then $\mathcal{P}=\pres{G,\mathbf{x}}{\mathbf{r}}$ is an ordinary group presentation with cellular model $L$ and asphericity of $\P$ is equivalent to (topological) asphericity of $L$ (see e.g.\ \cite{AJSAlgTop}). The conclusions of Theorem~\ref{Theorem:Theory} are relativized versions of standard facts about aspherical two-complexes. In particular, the fundamental group of an aspherical two-complex is torsion-free and has geometric and cohomological dimension at most two.
\end{re}

The asphericity status of a relative presentation $\P = \pres{G,\mathbf{x}}{\mathbf{r}}$ is unaffected if any relator $r \in \mathbf{r}$ is replaced by a conjugate $wr^\epsilon w^{-1}$ where $\epsilon = \pm 1$ and $w \in G \ast F$ because the homotopy type of the cellular model is unaffected by such a change. Thus one can assume that the relators of $\P$ are cyclically reduced. However, asphericity does impose some combinatorial restrictions on the relator set $\mathbf{r} \subseteq G \ast F$.

\begin{lemma}\label{lemma:relatorconditions}
Let $\P = \pres{G,\mathbf{x}}{\mathbf{r}}$ be an aspherical relative presentation.
\begin{itemize}
\item[(a)] No relator $r \in \mathbf{r}$ is conjugate in $G \ast F$ to an element of $G$.
\item[(b)] If $r,s \in \mathbf{r}$ and $s$ is conjugate to $r^\epsilon$ in $G \ast F$, then $r=s$ and $\epsilon = 1$.
\item[(c)] If $r \in \mathbf{r}$, then $r$ is not expressible as a proper power in $G \ast F$: if $r = \mathring{r}^e$ in $G \ast F$, then $e = \pm 1$.
\end{itemize}
\end{lemma}

\begin{proof} Let $(L,K) = (L(\P),K(G,1))$ and $K_1 = K \vee \bigvee_\mathbf{x} S^1_x \subseteq L$. Since $\pi_2L = 0$, the long exact sequence for the pair $(L,K_1)$ shows that the homotopy boundary $\partial: \pi_2(L,K_1) \ra \pi_1K_1$ is injective. For each $r \in \mathbf{r}$, the two-cell $c^2_r$ of $L$ has a characteristic map $\varphi^2_r:(B^2,S^1) \ra (L,K_1)$  with homotopy class $[\varphi_r] \in \pi_2(L,K_1)$. The Hurewicz homomorphism $h_2: \pi_2(L,K_1) \ra H_2(L,K_1)$ carries the corresponding homotopy classes $[\varphi_r] \in \pi_2(L,K_1)$, $r \in \mathbf{r}$, to a free $\Z$-basis for the cellular homology group $H_2(L,K_1) \cong \Z^\mathbf{r}$.

To prove (a), just suppose that $r = wgw^{-1} \in \mathbf{r}$ where $g \in G$ and $w \in G \ast F$. Using the homotopy action of $\pi_1K_1 \cong G \ast F$ on $\pi_2(L,K_1)$, the element $C = w^{-1}\cdot [\varphi_r] \in \pi_2(L,K_1)$ satisfies $\partial C = g \in G \cong \pi_1K$, so $C$ is in the image of the inclusion-induced homomorphism $\pi_2(L,K) \ra \pi_2(L,K_1)$. On the other hand $h_2(C) = h_2([\varphi_r]) \neq 0$, so this contradicts asphericity of $\P$.

For (b), if $r^\epsilon = ws w^{-1}$ in $G \ast F$, then  the element $\Delta = [\varphi_r]^{-\epsilon}(w \cdot [\varphi_s]) \in \pi_2(L,K_1)$ satisfies $\partial \Delta = r^{-\epsilon} wsw^{-1} = 1$ in $\pi_1K_1 \cong G \ast F$ and so $\Delta$ so is trivial in $\pi_2(L,K_1)$. Applying $h_2$, we therefore have the $\Z$-linear relation $0 = h_2(\Delta) = -\epsilon h_2([\varphi_r])+h_2([\varphi_s])$ involving members of a basis of $H_2(L,K_1) \cong \Z^\mathbf{r}$. This implies that $\epsilon = 1$ and $r = s$ in $\mathbf{r}$.

For (c), suppose that $r = \mathring{r}^{e} \in \mathbf{r} \subseteq G \ast F$. We show that $e = \pm 1$. The element $\Sigma = (\mathring{r} \cdot [\varphi_r])[\varphi_r]^{-1} \in \pi_2(L,K_1)$ satisfies $\partial \Sigma = \mathring{r}r\mathring{r}^{-1}r^{-1} = 1$ in $\pi_1K_1 \cong G \ast F$, so $\Sigma = 1$ in $\pi_2(L,K_1)$. Letting $p:\widetilde{L} \ra L$ be the simply connected covering space, a choice of basepoints determines an isomorphism $\pi_2(L,K_1) \ra \pi_2(\widetilde{L},p^{-1}(K_1))$, which we then compose with the Hurewicz homomorphism $h_2':\pi_2(\widetilde{L},p^{-1}(K_1)) \ra H_2(\widetilde{L},p^{-1}(K_1))$. In this way, the elements $[\varphi_r]$ pass to a free $\Z G(\P)$-basis for the equivariant cellular chain group $H_2(\widetilde{L},p^{-1}(K_1)) \cong \bigoplus_\mathbf{r} \Z G(\P)$ consisting of preferred lifts $\tilde{c}^2_r$ of the two-cells of $L-K$. Under this process, the element $\Sigma \in \pi_2(L,K)$ lifts to $(\mathring{r}-1)\tilde{c}^2_r \in H_2(\widetilde{L},p^{-1}(K_1))$, so the fact that $\Sigma$ is trivial in $\pi_2(L,K_1)$ implies that $\mathring{r} - 1 = 0$ in $\Z G(\P)$, whence $\mathring{r} = 1$ in $G(\P)$. Thus there exists $D \in \pi_2(L,K_1)$ such that $\partial(D) = \mathring{r} \in \pi_1K_1 \cong G \ast F$, so that the element $D^e [\varphi_r]^{-1}$ is in the kernel of the homotopy boundary $\partial$. Again, the fact that $\pi_2L = 0$ implies that $D^e [\varphi_r]^{-1}$ is trivial in $\pi_2(L,K_1)$ and so $0 = h_2(D^e [\varphi_r]^{-1}) = eh_2(D) - h_2([\varphi_r])$. Since $h_2([\varphi_r])$ is part of a $\Z$-basis for $H_2(L,K_1)$, this implies that $e = \pm 1$.
\end{proof}
\begin{de}[Orientability \cite{BP}]\label{def:orientable} A relative presentation $\P = \pres{G,\mathbf{x}}{\mathbf{r}}$ is \em orientable \em if it satisfies the relator conditions in Lemma~\ref{lemma:relatorconditions}(a) and (b).
\end{de}
Thus every aspherical relative presentation is orientable. The term `orientable' is used because it implies that no relator is a cyclic permutation of its inverse; in particular, no relator can be written in the form $r = wg_1w^{-1}g_2$ where $w \in G \ast F$ and $g_1, g_2 \in G$ both have order two. See \cite[Section~1.1]{BP}. Note that orientability does not exclude the presence of proper power relators. See Remark \ref{rem:genremarks}(c).

\begin{re}[Proper powers]\label{rem:caveats} In the presence of proper power relators $r = \mathring{r}^{e(r)} \in \mathbf{r} \subseteq G \ast F$, where $e(r)$ is the (maximal) exponent for $r$, an expanded cellular model $M(\P)$ is obtained from the one-point union $K(G,1) \vee \bigvee_\mathbf{x} S^1_x$ by attaching a copy of a $K(\Z_{e(r)},1)$-complex in place of the two-cell $c^2_r$ so that $K(G,1) \subseteq L(\P) \subseteq M(\P)$. The construction is described in detail in \cite[Section 4]{BP} and is based upon one that was introduced by Dyer and Vasquez \cite{DV73}. In \cite[Theorem 4.3]{BP} it was shown that if $\pi_2(M(\P),K(G,1))$ is trivial, then $M(\P)$ is a $K(G(\P),1)$-complex and that the root $\mathring{r}$ of each relator generates a cyclic subgroup $\sgp{\mathring{r}}$ of order $e(r)$ in $G(\P)$. (See \cite[Proposition 1]{Hueb} or \cite[page 36]{BP}.) Cohomology calculations and classifications of finite subgroups follow as in Theorem~\ref{Theorem:Theory}(b),(c), see \cite[(0.3),(0.4)]{BP} for details. The property $\pi_2(M(\P),K(G,1)) = 0$ for the relative presentation $\P$ is analogous to the concept of \em combinatorial asphericity \em for ordinary presentations \cite{CCH}. The long exact homotopy sequence for the triple $(M(\P), L(\P), K(G,1))$ implies that the inclusion-induced homomorphism $\pi_2(L(\P),K(G,1)) \ra \pi_2(M(\P),K(G,1))$ is surjective since $\pi_2(M(\P),L(\P))$ is trivial by the Cellular Approximation Theorem \cite[Theorem~4.8]{HatcherAT}. In this survey we focus on the stronger condition $\pi_2(L(\P),K(G,1)) = 0$ of Definition~\ref{Definition:Aspherical} in order to simplify the narrative and to emphasize relativization of topological asphericity for two-complexes, as in Theorem~\ref{Theorem:Theory}(a).
\end{re}

\subsection{Weak asphericity}\label{subsec:weak}

An alternative (weaker) asphericity concept that is not tied directly to the $\pi_1$-embedding question was introduced in~\cite{BBP} and studied further in~\cite{Davidson09},\cite{Ahmad},\cite{SKK6},\cite{SKK}. For the cellular model $(L,K) = (L(\P),K(G,1))$ of a relative presentation $\mathcal{P} = \pres{G,\mathbf{x}}{\mathbf{r}}$, we can assume that the two-skeleton $K^2$ is modeled on a presentation $\mathcal{Q}$ for $G$ so that the two-skeleton $L^2$ is modeled on a lifted ordinary presentation $\widehat{\mathcal{P}}$ for $G(\mathcal{P})$ \cite[Section~1.6]{BP}. The second homotopy modules $\pi_2K^2, \pi_2L^2$ have sometimes been denoted by $\pi_2\mathcal{Q}, \pi_2\widehat{\P}$ (e.g.~\cite{BPpi2gen}).

\begin{de}[See \cite{BBP},\cite{Davidson09}]\label{Definition:Weak} The relative group presentation $\mathcal{P}$ is \em weakly aspherical \em if the second homotopy module $\pi_2L^2 = \pi_2\widehat{\P}$ is generated as a $\Z G(\P)$-module by the image of the inclusion-induced homomorphism $\pi_2K^2 \ra \pi_2L^2$, or equivalently if the inclusion-induced homomorphism $\pi_2(K^2 \cup L^1) \ra \pi_2L^2$ is surjective.
\end{de}

We note that the term `weakly aspherical' was used with a related but slightly different meaning in~\cite{BP} -- see Remark~\ref{rem:genremarks}(a).

\begin{lemma}[{{See \textbf{\cite[Lemma~1.7]{BP},\cite[Theorem~4.2]{HowieHighPower}}}}]\label{Lemma:Weak}
Let $\mathcal{P} = \pres{G,\mathbf{x}}{\mathbf{r}}$ be a relative presentation and let $N$ be the kernel of the natural homomorphism $G \ra G(\P)$. Then:
\begin{enumerate}
\item[(a)] $\P$ is weakly aspherical if and only if $\pi_2L(\P) = 0$.
\item[(b)] If $\P$ is aspherical, then $\P$ is weakly aspherical.
\item[(c)] $L(\P)$ is a $K(G(\P),1)$-complex if and only if $\P$ is weakly aspherical and $H_k(N,\Z) = 0$ for $k \geq 3$.
\item[(d)] If $G$ embeds in $G(\P)$, then $\P$ is weakly aspherical if and only if $\P$ is aspherical.
\end{enumerate}
\end{lemma}

\begin{proof}
Letting $(L,K) = (L(\P),K(G,1))$, the statement (a) follows from a diagram chase in the following ladder of long exact homotopy sequences.
\[
\minCDarrowwidth25pt\begin{CD}
{\pi_2(K^2 \cup L^1)} @>>>{\pi_2L^2} @>>> {\pi_2(L^2,K^2 \cup L^1)} @>>> {\pi_1(K^2 \cup L^1)} @>>> {\pi_1L^2} \\
@. @ VVV @ V{\cong}VV @ V{\cong}VV @V{\cong}VV\\
{0} @>>>{\pi_2L} @>>> {\pi_2(L,K \cup L^1)} @>>> {\pi_1(K \cup L^1)} @>>> {\pi_1L} \\
\end{CD}
\]
That the inclusion-induced map $\pi_2(L^2,K^2 \cup L^1) \ra \pi_2(L,K \cup L^1)$ is an isomorphism follows from Whitehead's description of the relative group $\pi_2(Y,X)$ as a free crossed $\pi_1X$-module when $Y$ is obtained from $X$ by attaching two-cells \cite[Theorem~2.9]{AJSAlgTop},\cite{JHCCH2}. Note that the inclusion-induced homomorphism $\pi_2L^2 \ra \pi_2L$ is surjective by the Cellular Approximation Theorem~\cite[Theorem~4.8]{HatcherAT}. With Lemma~ \ref{Lemma:KG1}, the statement (b) is now obvious. Looking more closely at the proof of Lemma~\ref{Lemma:Covering}, if $p:\widetilde{L} \ra L$ is the universal covering projection, then each component of the pre-image $p^{-1}(K)$ has fundamental group isomorphic to $N$ and so is a $K(N,1)$-complex because $K$ is aspherical. Thus $H_k\widetilde{L} \cong H_k(p^{-1}(K))$ is a direct sum of copies of $H_k(N,\Z)$ for $k \geq 3$; the statements (c) and (d) follow at once.
\end{proof}

\begin{re}\label{rem:weakrelators} The proofs of Lemma~\ref{lemma:relatorconditions}(b),(c) used only the fact that $\pi_2L(\P) = 0$, so those results also apply to weakly aspherical relative presentations by \linebreak Lemma~\ref{Lemma:Weak}(a). However, Lemma~\ref{lemma:relatorconditions}(a) does not hold for all weakly aspherical presentations. For example, if $G = \sgp{g} \cong \Z$, then the relative presentation $\P = \pres{G,\emptyset}{g}$ is weakly aspherical because $L(\P)$ is a disc. However $\P$ is non-aspherical because the natural homomorphism $G \ra G(\P)$ is not injective. Thus, asphericity and weak asphericity are distinct concepts.
\end{re}

The next example also distinguishes asphericity from weak asphericity.

\begin{ex}[Weakly aspherical but non-aspherical.]\label{ex:WeakAspNotAsp}
Let $G = \gpres{g}\cong\Z$ and let $\mathcal{P} = \pres{G,x}{gxg^{-1}x^{-2}, xgx^{-1}g^{-2}}$. Here $G(\mathcal{P})$ is the trivial group and $L(\mathcal{P})$ is a contractible two-complex. So $\mathcal{P}$ is weakly aspherical by Lemma~\ref{Lemma:Weak}(a). Since $G\ra G(\mathcal{P})$ is not injective, $\P$ is non-aspherical by Lemma~\ref{Lemma:KG1}.
\end{ex}

\begin{re}\label{rem:KervaireQ}
An open question from the 1950's asks if there is a one-relator relative presentation $\P = \pres{G,x}{x^{e_1}g_1\ldots x^{e_\gwell}g_\gwell}$ with $G(\P) = 1$ and $G \neq 1$. This question has been attributed to Kervaire and to Laudenbach, see \cite[Chapter I.6]{LyndonSchupp} and \cite[page 403]{MKS}. Freedman \cite{Freedman78} noted that in any such example the exponent sum $\sum_{i=1}^l e_i$ is equal to $\pm 1$ and that the group $G$ is perfect and has no proper (normal) subgroups of finite index (by a result of Gerstenhaber and Rothaus \cite{GR}). Klyachko proved that the group $G$ must have torsion \cite{Kly} and can be chosen to be simple \cite{Kly05}. Arguing as in Lemma~\ref{Lemma:Weak}(c), if we further assume that $L(\P)$ is contractible, then $G$ is acyclic.
\end{re}

\subsection{Preliminary reductions}

If the natural homomorphism $G \ra G(\P)$ is injective, then the long exact homotopy sequence for the pair $(L,K) = (L(\P),K(G,1))$ shows that $\pi_2L \cong \pi_2(L,K)$, which means that elements of the relative homotopy group are represented by spherical maps $S^2 \ra L$. The implicit involvement of both $\pi_1$ and $\pi_2$ in Definition~\ref{Definition:Aspherical} (i.e.\ Lemma~\ref{Lemma:KG1}) also means that we can reduce to the case where the coefficient group is generated by the occurrences of coefficients in the relative relators.

\begin{lemma}\label{Lemma:Coefficient} Given a relative presentation $\mathcal{P} = \pres{G,\mathbf{x}}{\mathbf{r}}$, if $H$ is any subgroup of $G$ for which the set of relators $\mathbf{r} \subseteq G \ast F$ is contained in the free product $H\ast F$, then $\mathcal{P}$ is aspherical if and only if $\mathcal{P}_0 = \pres{H,\mathbf{x}}{\mathbf{r}}$ is aspherical.
\end{lemma}

\begin{proof}
Since $H \leq G$, we can construct a $K(G,1)$-complex $K$ that contains a $K(H,1)$-complex $K_0$ as a subcomplex and then the cellular models $(L,K)$ and $(L_0,K_0)$ are such that $L = K \cup_{K_0} L_0$. Since the inclusion of $K_0$ in $K$ induces the monomorphism $H \ra G$, it follows that if either $\mathcal{P}_0$ or $\mathcal{P}$ is aspherical, then the natural maps $H \ra G(\P_0)$ and $G \ra G(\P)$ are both injective and there is an amalgamated free product decomposition $G(\P) = G \ast_{H} G(\P_0)$; see e.g.\ \cite[page 89]{Howie83}. We may therefore assume that $H \ra G(\P_0)$, $G \ra G(\P)$, and $G(\P_0) \ra G(\P)$ are all injective. Letting $p:\widetilde{L} \ra L$ be the universal covering projection, the Mayer-Vietoris sequence for the union $\widetilde{L} = p^{-1}(K) \cup_{p^{-1}(K_0)} p^{-1}(L_0)$ shows that $\pi_2L \cong H_2\widetilde{L} \cong H_2 (p^{-1}(L_0))$ is a direct sum of copies of $H_2\widetilde{L_0} \cong \pi_2L_0$ and so the result follows from Lemma~\ref{Lemma:Weak}. \end{proof}

For one-relator relative presentations $\P = \pres{G,x}{r}$, many asphericity classifications can be reduced to case where the exponents occurring on $x$ are relatively prime.

\begin{lemma}\label{lem:relativelyprime} Suppose that $|d| \geq 2$ and consider the relative presentations $\P = \pres{G,x}{r}$ and $\mathcal{V} = \pres{G,y}{s}$ where $r = x^{de_1}g_1 \ldots  x^{de_\ell}g_\ell$ and $s = y^{e_1}g_1 \ldots  y^{e_\ell}g_\ell$. Then $\P$ is aspherical if and only if $\mathcal{V}$ is aspherical and the element $y$ has infinite order in $G(\mathcal{V})$.\footnote{See Example~\ref{ex:finiteorderx} and Question~\ref{qu:isomorphism} below.}
\end{lemma}

\begin{proof}
There are homotopy equivalences
\begin{eqnarray*}
L(\P) &=& K(G,1) \vee S^1_x \cup c^2_r\\
&\simeq& K(G,1) \vee S^1_x \vee S^1_y \cup c^2_r \cup c^2_{y=x^d}\\
&\simeq& K(G,1) \vee S^1_x \vee S^1_y \cup c^2_s \cup c^2_{y=x^d}\\
&=& L(\mathcal{V}) \vee S^1_x \cup c^2_{y=x^d}
\end{eqnarray*}
showing that $G(\mathcal{P}) = \pi_1L(\mathcal{P})$ is obtained from $G(\mathcal{V}) = \pi_1L(\mathcal{V})$ by adjoining a $d$th root of $y \in G(\mathcal{V})$ \cite[Theorem 5.1]{NeumannRoot}. In particular, $G(\P) \cong G(\mathcal{N})$ where $\mathcal{N}$ is the one-relator relative presentation $\mathcal{N} = \pres{G(\mathcal{V}),x}{y=x^d}$ and so we have an amalgamated free product decomposition $G(\P) \cong G(\mathcal{N}) \cong G(\mathcal{V}) \ast_{y=x^d} \sgp{x}$. Thus we can assume that $G$ embeds in both $G(\mathcal{V})$ and in $G(\P)$.

Suppose first that $\P$ is aspherical. It follows from Lemmas \ref{Lemma:Covering} and \ref{Lemma:KG1} that both $\mathcal{V}$ and $\mathcal{N}$ are aspherical. Further, since $|d| \geq 2$, the element $x$ is not conjugate in $G(\P) \cong G(\mathcal{N})$ to an element of $G(\mathcal{V})$, so Theorem~\ref{Theorem:Theory}(c) implies that $x$ has infinite order in $G(\mathcal{N})$ and so $y$ has infinite order in $G(\mathcal{V})$.

Conversely, if $\mathcal{V}$ is aspherical and $y$ has infinite order in $G(\mathcal{V})$, then we can prove that $\mathcal{N}$ is aspherical using the weight test \cite[Theorem 2.1]{BP} (see Theorems \ref{thm:DRAspherical} and \ref{Theorem:WeightTest} below) applied to the star graph $\mathcal{N}^{\mathrm{st}}$ of $\mathcal{N}$: apply the weight $-1$ to the edge labeled $y$ and apply the weight $1$ to all remaining edges in the star graph. Thus,
$$
L(\P) \simeq L(\mathcal{V}) \vee S^1_x \cup c^2_{y=x^d} = K(G(\mathcal{V}),1) \vee S^1_x \cup c^2_{y=x^d} = L(\mathcal{N})
$$
and so $\P$ is aspherical.
\end{proof}

In the case of \em balanced \em  relative presentations (i.e.\ relative presentations $\mathcal{P} = \pres{G,\mathbf{x}}{\mathbf{r}}$ with $|\mathbf{x}|=|\mathbf{r}|$) a simple instance of asphericity occurs when $G = G(\P)$.

\begin{lemma}\label{Lemma:G=H} Consider a balanced relative presentation $\mathcal{P} = \pres{G,\mathbf{x}}{\mathbf{r}}$ for a group $G(\P)$ where $|\mathbf{x}| = |\mathbf{r}| < \infty$. If the natural homomorphism $G \ra G(\mathcal{P})$ is an isomorphism, then $\mathcal{P}$ is aspherical.
\end{lemma}

\begin{proof}
Let $(L,K)$ be the cellular model of $\mathcal{P}$ where $K = K(G,1)$. Assuming that $G \ra G(\mathcal{P})$ is an isomorphism, we show that $L$ is aspherical, so $\pi_2(L,K) \cong \pi_2L = 0$ and so $\mathcal{P}$ is aspherical.

Let $p: \widetilde{L} \ra L$ be the universal covering projection. The pre-image $p^{-1}(K)$ is connected (since the inclusion-induced homomorphism $G = \pi_1K \ra \pi_1L = G(\mathcal{P})$ is surjective) and simply connected (since $\pi_1K \ra \pi_1L$ is injective), so $p^{-1}(K) = \widetilde{K}$ is contractible. We therefore have a homotopy equivalence $\widetilde{L} \simeq \widetilde{L}/\widetilde{K} = X$ where $X$ is obtained by identifying $p^{-1}(K)$ to a zero-cell $c^0$. We show that $X$, and hence $\widetilde{L}$, is contractible, so that $L$ is aspherical. Note that $X$ is simply connected and two-dimensional, so it suffices to show that $H_2X = 0$.

The group of deck transformations $\mathrm{Aut}(p) \cong G(\mathcal{P})$ freely permutes the cells of $\widetilde{L}$ leaving the subcomplex $p^{-1}(K) = \widetilde{K}$ invariant. Thus $G(\P)$ acts freely and cellularly on $X-c^0$ and so the cellular chain complex of $X$ consists of $\Z G(\P)$-modules
$$
C_2 X \stackrel{\partial_2}{\ra} C_1 X \stackrel{\partial_1}{\ra} C_0 X
$$
that are free of rank $n = |\mathbf{x}| = |\mathbf{r}|$ in dimensions one and two and where $C_0 X = \Z$ has trivial $G(\mathcal{P})$-action. The boundary operator $\partial_1$ is trivial because $X$ has just a single zero-cell. The fact that $X \simeq \widetilde{L}$ is simply connected implies that the boundary operator $\partial_2: \Z G(\mathcal{P})^n \ra \Z G(\mathcal{P})^n$ is a surjective endomorphism of the free $\Z G(\P)$-module of rank $n$, which therefore has a right inverse: $\partial_2 \circ \sigma = 1$. A theorem of Kaplansky \cite{Kaplansky72},\cite{Montgomery69} then implies that $\sigma \circ \partial_2 = 1$ and so $H_2X = \ker \partial_2 = 0$.
\end{proof}

In connection with Theorem~\ref{Theorem:Theory}(c), we note that non-obvious torsion can occur in the aspherical setting.

\begin{ex}[Torsion and asphericity]\label{ex:finiteorderx}
The group defined by the relative presentation $\mathcal{J} = \pres{\pres{g}{g^4},x}{x^4gx^{-3}g^2}$ (denoted $J_4(1,-3)$ in~\cite{BW1}) is cyclic of order four and $g = x$ in $G(\mathcal{J})$. Thus $\mathcal{J}$ is aspherical by Lemma~\ref{Lemma:G=H}, even though $x$ is a finite order element.
\end{ex}

In light of this example we pose the following question.

\begin{qu}\label{qu:isomorphism}If $\P = \pres{G,x}{x^{e_1}g_1\cdots x^{e_\ell}g_\ell}$ is an aspherical one-relator relative presentation and $x$ determines an element of finite order in $G(\P)$, does it follow that $G \ra G(\P)$ is an isomorphism? Note that the exponent sum $\sum_{i=1}^\ell e_i$ must be equal to $\pm 1$, for otherwise the element $x$ lies outside the normal closure of the natural image of $G$ in $G(\P)$ so $x$ has infinite order by Theorem~\ref{Theorem:Theory}(c).\end{qu}

\subsection{Three-manifolds and zero divisors}\label{subsec:three}

A \em three-manifold group \em  is the fundamental group of a (PL) three-manifold, which need not be compact, closed, or orientable. A result of Ratcliffe \cite{Ratcliffe87} concerning rational Euler characteristics for groups \cite{Wall61} was exploited in \cite{HW16} to establish a connection between relative asphericity and the three-manifold status for the group $G(\P)$. We make this connection explicit in the following result.

\begin{thm}[{{Three-manifold criterion}}]\label{thm:3mfdcriterion}  Suppose that the relative presentation $\mathcal{P} = \pres{G,\mathbf{x}}{\mathbf{r}}$ is such that $G, \mathbf{x}$, and $\mathbf{r}$ are finite and the natural homomorphism $G \ra G(\P)$ is split by a retraction $\nu: G(\P) \ra G$. Assume that $G(\P)$ is a virtual three-manifold group. Then:
\begin{enumerate}
\item[(a)] If $\P$ is aspherical, then $|\mathbf{r}| \leq |\mathbf{x}|$; and
\item[(b)] If $|\mathbf{r}|=|\mathbf{x}|$ (i.e.\ $\mathcal{P}$ is balanced), then $\P$ is aspherical $\iff G \ra G(\P)$ is an isomorphism.
\end{enumerate}
\end{thm}

\begin{proof}
Assume that $\P$ is aspherical. With notation as in the proof of Theorem~\ref{Theorem:Theory}(c), $p:\overline{L} \ra L$ is a regular covering projection with $\mathrm{Aut}(p) \cong G$ and the pre-image $p^{-1}(K(G,1)) = \widetilde{K(G,1)}$ is contractible. Thus $X = \overline{L}/\widetilde{K(G,1)} \simeq \overline{L} \simeq K(\Gamma,1)$ where $\Gamma = \ker \nu$. Asphericity of $X$ implies that $\chi(X) = 1-|G|(|\mathbf{x}|-|\mathbf{r}|)$ is an invariant of $\pi_1X \cong \pi_1\overline{L} \cong \Gamma$, denoted $\chi(\Gamma)$ as in \cite{Wall61}. If $G(\P)$, and hence $\Gamma$, is a virtual three-manifold group, then \cite[Theorem~2(i)]{Ratcliffe87} implies that $\chi(\Gamma) \leq 1$ so $|\mathbf{r}| \leq |\mathbf{x}|$, as in (a). If $|\mathbf{r}|=|\mathbf{x}|$, as in (b), then $\chi(\Gamma) = 1$, so \cite[Theorem~2(ii)]{Ratcliffe87} implies that $\Gamma = 1$, in which case $G \ra G(\P)$ is an isomorphism.
\end{proof}

\begin{ex}[Virtual three-manifold groups]\label{ex:3mfdgps}
(a) Let $G=\pres{g}{g^n}$ where $n\geq 6$ is even and let $\mathcal{P}=\pres{G,x}{x^2g^2x^{-1}g^{-1}}$.
It is well known (see, for example, \cite[page~196]{JBook2}) that $G(\mathcal{P})\cong F(2,n)\rtimes \Z_n$ where
\[F(2,n)=\pres{x_0,\ldots ,x_{n-1}}{x_ix_{i+1}=x_{i+2}\ (0\leq i\leq n-1)}\]
is a \em Fibonacci group. \em By~\cite{HKM} we have that $F(2,n)$ is an (infinite) three-manifold group, so $G(\mathcal{P})$ is a virtual three-manifold group and $G(\mathcal{P})\not \cong \pres{g}{g^n}$. Thus $\P$ is non-aspherical by Theorem~\ref{thm:3mfdcriterion}.

(b) Let $G=\pres{g}{g^n}$ where $n\geq 6$ and let $\mathcal{P}=\pres{G,x}{x^2gx^{-1}g}$. As in part~(a), we have that $G(\mathcal{P})\cong S(2,n)\rtimes \Z_n$ where
\[S(2,n)=\pres{x_0,\ldots ,x_{n-1}}{x_ix_{i+2}=x_{i+1}\ (0\leq i\leq n-1)}\]
is a \em Sieradski group. \em By~\cite{AJSSquash},\cite{CHK} we have that $S(2,n)$ is an (infinite) three-manifold group, so again we have that $\mathcal{P}$ is non-aspherical.
\end{ex}

Ivanov~\cite{Ivanov1999} discovered a direct connection between non-weak asphericity for a one-relator relative presentation $\P$ and the existence of zero divisors in the group ring $\Z G(\P)$. The following generalization of Ivanov's result is due to Leary~\cite{Leary2000}.

\begin{thm}[\cite{Leary2000}]\label{thm:zero divisor} Let $\P = \pres{G,x}{x^{e_1}g_1\cdots x^{e_\ell}g_\ell}$ be a one-relator relative presentation. Assume that the kernel of the natural homomorphism $G \ra G(\P)$ is acyclic and that the exponent sum $\sum_{i=1}^\ell e_i$ is nonzero. If $\P$ is non-weakly aspherical (in the sense of Definition~\ref{Definition:Weak}), then the group ring $\Z G(\P)$ contains a non-trivial zero divisor that is a $\Z$-linear combination of at most $\sum_{i=1}^\ell |e_i|$ elements of $G(\P)$.
\end{thm}

Ivanov conjectured that if $G$ is the fundamental group of an aspherical two-complex and the natural homomorphism $G \ra G(\P)$ is injective, then $\P$ is aspherical (in the sense of Definition~\ref{Definition:Aspherical}) if and only if $G(\P)$ is torsion-free~\cite{Ivanov1999}. A counterexample would require a torsion-free group whose integral group ring has a non-trivial zero divisor. That no such group exists is a longstanding open conjecture that is widely attributed to Kaplansky. Kim \cite{SKK6},\cite{SKK} has obtained results on the asphericity of certain one-relator relative presentations with torsion-free coefficient group.

\section{Methods from combinatorial geometry}\label{section:Methods}

A combinatorial geometric form of asphericity for ordinary presentations and their two-dimensional cellular models was described by Sieradski \cite[Section~4]{AJSColor}, who gave a practical `coloring test' for detecting the property. Gersten \cite{Gersten87} termed the property \textit{diagrammatic reducibility} (DR); he developed a flexible `weight test' generalizing the coloring test and applied the DR property to the study of equations over groups. For a relative presentation $\P = \pres{G,\mathbf{x}}{\mathbf{r}}$, any element of the kernel of the natural homomorphism $G \ra G(\P)$ leads to a map of pairs $(B^2,S^1) \ra (L(\P),K(G,1))$. In \cite[Lemma 1]{Howie83}, Howie applied topological methods to show that any such map of pairs gives rise to a `relative diagram' that emerges as the geometric dual to naturally occurring combinatorial structure called a `picture'. In this section we examine combinatorial geometric conditions on pictures that are sufficient to guarantee asphericity of $\P$ in the sense of Definition~\ref{Definition:Aspherical}.

\subsection{Diagrammatic reducibility}

A \em picture \em  $\mathbb{P}$ is a finite collection of pairwise disjoint discs $\{D_{1}, \ldots, D_{m}\}$ in the interior of a disc $D^{2}$, together with a finite collection of pairwise disjoint simple arcs $\{\alpha_{1},\ldots,\alpha _{n}\}$ embedded in the closure of $D^{2} - \bigcup_{i=1}^{m} D_{i}$ in such a way that each arc meets $\partial D^{2} \cup \bigcup_{i=1}^{m} D_{i}$ transversely in its end points \cite{BP}. The discs of a picture $\mathbb{P}$ are also called \em vertices \em  of $\mathbb{P}$ and for this reason $v_{i}$ or simply $v$ is often used in place of $D_{i}$. The \em boundary \em  of $\mathbb{P}$ is the circle $\partial D^{2}$, denoted by $\partial \mathbb{P}$. For $1\leq i \leq m$, the \em corners \em  of $D_{i}$ are the closures of the connected components of $\partial D_{i}- \bigcup_{j=1}^{n} \alpha_{j}$, where $\partial D_{i}$ is the boundary of $D_{i}$. The \em regions \em  $\Delta$ of $\mathbb{P}$ are the
closures of the connected components of $D^{2}-\big(\bigcup_{i=1}^{m} D_{i} \cup \bigcup_{j=1}^{n} \alpha_{j} \big)$. An \em inner region \em of $\mathbb{P}$ is a simply connected region of $\mathbb{P}$ that does not meet $\partial \mathbb{P}$.  (If $\mathbb{P}$ is connected note that \em any \em region that does not meet the boundary is simply connected.)

The picture  $\mathbb{P}$ is \em  non-trivial \em if $m\geq1$, is \em connected \em  if $\bigcup_{i=1}^{m} D_{i} \cup \bigcup_{j=1}^{n} \alpha_{j}$ is connected, and is \em spherical \em  if it is non-trivial and if none of the arcs meets the boundary of $\mathbb{P}$.
The number of arcs or \em edges \em  in $\partial\Delta$ is called the \em degree \em of the region $\Delta$ and is denoted by $d(\Delta)$. A region of degree $k$ will be called a \em $k$-gonal region. \em

Now consider the relative presentation   $\pres{G,\mathbf{x}}{\mathbf{r}}$ and suppose that the picture $\mathbb{P}$ is labelled in the following sense: each arc $\alpha_{j}$ is equipped with  a normal orientation, indicated by a short arrow meeting the arc transversely, and labelled by an element of $\mathbf{x}\cup \mathbf{x}^{-1}$. Each corner of $\mathbb{P}$ is oriented \em clockwise \em  (relative to the disc/vertex to which the corner is associated) and labelled by an element of $G$. (We remark here that the convention \em anti-clockwise \em has been adopted by some authors.) If $\kappa$ is a corner of a disc $D_{i}$ of $\mathbb{P}$, then $W(\kappa)$ will be the word obtained by reading in a  clockwise order the labels on the arcs and corners meeting $\partial D_{i}$ beginning with the label on the first arc we meet as we  read the clockwise corner $\kappa$. If we cross an arc labelled $x$ in the direction of its normal orientation, we read $x$, else we read $x^{-1}$.

A \em picture over $\mathcal{P}$ \em  is a picture $\mathbb{P}$ labelled in such a way the following are satisfied:

\begin{itemize}
  \item[1.] For each corner $\kappa$ of  $\mathbb{P}$, $ W(\kappa) \in
\mathbf{r}^{*}$, the set of all cyclic permutations of $\mathbf{r}\cup
\mathbf{r}^{-1}$ which begin with a member of  $\mathbf{x}$.
  \item[2.] If $g_{1},\ldots,g_{l}$ is the sequence of corner labels encountered
in an \em anticlockwise \em  traversal of the boundary of an inner region
$\Delta$ of $\mathbb{P}$, then the product $g_{1}g_{2}\ldots g_{l}=1$
in $G$. We say that $g_{1}g_{2}\ldots g_{l}$ is the \em label \em  of $\Delta$.
\end{itemize}

A \em dipole \em  in a labelled picture  $\mathbb{P}$ over $\mathcal{P}$ consists of  corners $\kappa$ and $\kappa '$ of  $\mathbb{P}$ together with an arc joining the two corners such that $\kappa$ and $\kappa'$ belong to the same region and such that if $W(\kappa)$= $Sg$ where $g\in G$ and $S$ begins and ends with a member of $\mathbf{x}\cup \mathbf{x}^{-1}$, then $W(\kappa')$= $S^{-1}g^{-1}$. For example if $\mathbf{x}=\{x\}$ and $\mathbf{r}=\{x^3gx^{-1}h\}$ then a dipole is given in Figure~\ref{fig:dipoleandvertex}(a).

\begin{de}[\cite{BP},\cite{Davidson09}]\label{def:DR}
A relative presentation $\mathcal{P}$ is called \em diagrammatically reducible \em  if every connected spherical picture over $\mathcal{P}$ contains a dipole.
\end{de}

A picture $\mathbb{P}$ is called \em reduced \em  if it does not contain a dipole. Thus if $\mathcal{P}$ fails to be diagrammatically reducible, or is \em non-diagrammatically reducible, \em then there is a non-trivial reduced, connected, spherical picture over $\mathcal{P}$. A connected spherical picture $\mathbb{P}$ is called \em strictly spherical \em  if the product of the corner labels in the annular region, taken in a clockwise direction, is equal to the identity in $G$.

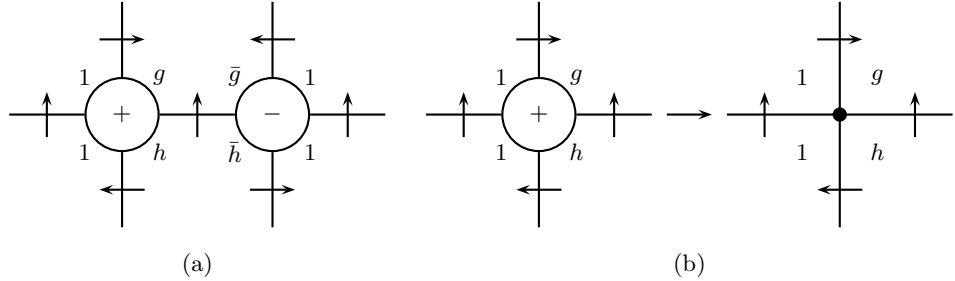
\begin{figure}
\begin{tabular}{l@{\qquad\qquad}l}
  \begin{pspicture}(0,1)(7,5)
  \footnotesize
  \psset{xunit=1cm}
  \psset{yunit=1cm}
\pscircle(2,3){0.5}
\rput(2,3){$+$}
\pscircle(4,3){0.5}
\rput(4,3){$-$}
\psline(0.5,3)(1.5,3)
\rput(1.5,3.5){$1$}
\rput(1.5,2.5){$1$}
\psline[arrowscale=1.2]{->}(1,2.7)(1,3.3)
\psline(2.5,3)(3.5,3)
\rput(2.5,3.5){$g$}
\rput(2.5,2.5){$h$}
\psline[arrowscale=1.2]{->}(3,2.7)(3,3.3)
\psline(4.5,3)(5.5,3)
\rput(4.5,3.5){$1$}
\rput(4.5,2.5){$1$}
\psline[arrowscale=1.2]{->}(5,2.7)(5,3.3)
\rput(3.5,3.5){$\bar{g}$}
\rput(3.5,2.5){$\bar{h}$}
\psline(2,3.5)(2,4.5)
\psline[arrowscale=1.2]{->}(1.7,4)(2.3,4)
\psline(2,2.5)(2,1.5)
\psline[arrowscale=1.2]{->}(2.3,2)(1.7,2)
\psline(4,3.5)(4,4.5)
\psline[arrowscale=1.2]{->}(4.3,4)(3.7,4)
\psline(4,2.5)(4,1.5)
\psline[arrowscale=1.2]{->}(3.7,2)(4.3,2)
\rput(3,1){(a)}
\end{pspicture}
&
  \begin{pspicture}(3,1)(8,5)
  \footnotesize
  \psset{xunit=1cm}
  \psset{yunit=1cm}
\pscircle(2,3){0.5}
\rput(2,3){$+$}
\psline(0.5,3)(1.5,3)
\rput(1.5,3.5){$1$}
\rput(1.5,2.5){$1$}
\psline[arrowscale=1.2]{->}(1,2.7)(1,3.3)
\psline(2.5,3)(3.5,3)
\rput(2.5,3.5){$g$}
\rput(2.5,2.5){$h$}
\psline[arrowscale=1.2]{->}(3,2.7)(3,3.3)
\psline(2,3.5)(2,4.5)
\psline[arrowscale=1.2]{->}(1.7,4)(2.3,4)
\psline(2,2.5)(2,1.5)
\psline[arrowscale=1.2]{->}(2.3,2)(1.7,2)

\psline[arrowscale=1.2]{->}(3.7,3)(4.3,3)

%
\psdot[dotscale=1.5](6,3)
\rput(6,3){$-$}
\psline(6,3)(7.5,3)
\rput(6.5,3.5){$g$}
\rput(6.5,2.5){$h$}
\psline[arrowscale=1.2]{->}(7,2.7)(7,3.3)
\rput(5.5,3.5){$1$}
\rput(5.5,2.5){$1$}
\psline(6,3)(6,4.5)
\psline[arrowscale=1.2]{->}(5.7,4)(6.3,4)
\psline(6,3)(6,1.5)
\psline[arrowscale=1.2]{->}(6.3,2)(5.7,2)
\psline(4.5,3)(6,3)
\psline[arrowscale=1.2]{->}(5,2.7)(5,3.3)
\rput(4,1){(b)}
\end{pspicture}
\end{tabular}
      \caption{(a) dipole; (b) disc/vertex contraction.\label{fig:dipoleandvertex}}
      \end{figure}

\begin{de}\label{def:weakDR}
A relative presentation $\mathcal{P}$ is called \em weakly diagrammatically reducible \em  if every
connected strictly spherical picture over $\mathcal{P}$ contains a dipole.
\end{de}

\begin{re}\label{rem:genremarks}
 \begin{itemize}
  \item[(a)] Our definition of (weak) diagrammatically reducibility for a relative presentation coincides with (weak) asphericity as defined in~\cite{BP}.
  \item[(b)] It is shown in~\cite[Lemma~1.7]{BP} that if $\mathcal{P}$ is weakly diagrammatically reducible and the natural map of $G$ into $G(\mathcal{P})$ is injective then $\mathcal{P}$ is diagrammatically reducible.
  \item[(c)] Diagrammatically reducible relative presentations can have proper power relators. See Remark~\ref{rem:caveats} and Theorem~\ref{thm:DRAspherical}(a) below.
  \end{itemize}
  \end{re}

For two-complexes $X$, including two-dimensional cellular models of ordinary group presentations, the DR property (in the sense of \cite{AJSColor},\cite{Gersten87}) implies topological asphericity ($\pi_2X = 0$). For relative presentations that are orientable in the sense of Definition~\ref{def:orientable}, the relationship between diagrammatic reducibility and asphericity is a consequence of \cite[Theorem~4.1]{BP}.

\begin{thm}[{{See~\cite[Theorem~4.1]{BP}}}]\label{thm:DRAspherical} Let $\P = \pres{G,\mathbf{x}}{\mathbf{r}}$ be an orientable relative presentation.
\begin{itemize}
\item[(a)] If $\P$ is diagrammatically reducible, then $\pi_2(M(\P),K(G,1)) = 0$, where $M(\P)$ is the expanded cellular model described in Remark~\ref{rem:caveats}, and so the natural homomorphism $G \ra G(\P)$ is injective.

\item[(b)] If $\P$ is diagrammatically reducible (resp.\,weakly diagrammatically reducible) and has no proper power relators, then $\P$ is aspherical (resp.\,weakly aspherical). \end{itemize}
\end{thm}
The orientability condition in Theorem~\ref{thm:DRAspherical} is necessary to ensure that dipoles represent homotopically trivial elements of $\pi_2(L(\P),K(G,1))$. See \cite[Figure 1]{BP} for an example of a dipole in the non-orientable setting and see \cite[Example 3.5]{BShift} for additional discussion.

Diagrammatic reducibility and asphericity are distinct concepts for both ordinary and relative group presentations.

\begin{ex}[Aspherical but non-diagrammatically reducible]\label{ex:asphnotDR}
(a) If the group $G=\pres{g,h}{h^2=1,g=h^2}\cong \Z_2$ and $\mathcal{P} = \pres{G,x}{x^2gx^{-1}h}$, then the natural homomorphism $G \ra G(\mathcal{P})$ is an isomorphism, so $\mathcal{P}$ is aspherical. However, a reduced spherical picture $\mathbb{P}$ over $\mathcal{P}$ is depicted in Figure~\ref{fig:pequals1}, so $\mathcal{P}$ is non-diagrammatically reducible. Cyclically reducing, we obtain $\mathcal{P'} = \pres{G,x}{xh}$, which is obviously diagrammatically reducible. Thus diagrammatic reducibility is not robust under cyclic reduction of relators.
\begin{figure}
  \begin{center}
\begin{pspicture}(-0.1,1.5)(5.9,3)
  \footnotesize
  \psset{xunit=0.7cm}
  \psset{yunit=0.7cm}

\psdot(1,3)
\psdot(3,3)
\psdot(5,3)
\psdot(7,3)
\psline(1,3)(3,3)
\psline(5,3)(7,3)
\pscircle(0,3){0.7}
\pscircle(4,3){0.7}
\pscircle(8,3){0.7}
\psline[arrowscale=1.2]{->}(-1.3,3)(-0.7,3)
\psline[arrowscale=1.2]{->}(9.3,3)(8.7,3)
\psline[arrowscale=1.2]{->}(2,2.7)(2,3.3)
\psline[arrowscale=1.2]{->}(6,3.3)(6,2.7)
\psline[arrowscale=1.2]{->}(4,3.7)(4,4.3)
\psline[arrowscale=1.2]{->}(4,2.3)(4,1.7)
\rput(0.7,3){$\bar{g}$}
\rput(1.3,2.7){$\bar{h}$}
\rput(1.3,3.3){$1$}
\rput(3.3,3){$h$}
\rput(2.7,3.3){$g$}
\rput(2.7,2.7){$1$}
\rput(4.7,3){$h$}
\rput(5.3,2.7){$g$}
\rput(5.3,3.3){$1$}
\rput(7.3,3){$\bar{g}$}
\rput(6.7,3.3){$\bar{h}$}
\rput(6.7,2.7){$1$}
\end{pspicture}
\end{center}
  \caption{Reduced spherical picture.\label{fig:pequals1}}
\end{figure}
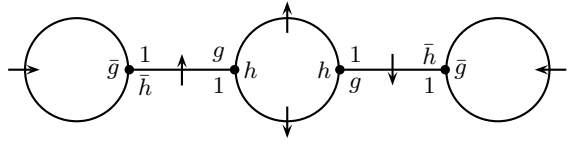

(b) Let $G = 1$ and $\P = \pres{G,x,y}{xyx^{-1}y^{-2}, yxy^{-1}x^{-2}}$. Taking $K = K(G,1)$ to be a point, the complex $L(\P)$ is a contractible, non-collapsible two-complex, so the relative presentation $\P$ is aspherical by Lemma~\ref{Lemma:G=H}. On the other hand, the result \cite[Theorem 2.4]{CorsonTrace} of Corson and Trace implies that $L(\P)$ is not DR (in the sense of \cite{AJSColor},\cite{Gersten87}). This in turn implies that the relative presentation $\P$ is non-diagrammatically reducible in the sense of Definition~\ref{def:DR}.
\end{ex}

In the next section we will discuss combinatorial geometric methods for detecting diagrammatic reducibility, as this concept is significant as a sufficient condition for asphericity of relative presentations. The next result shows that the diagrammatic reducibility concept also has independent group-theoretic and combinatorial significance.

\begin{thm}\label{thm:DRinf} Let $\P = \pres{G,\mathbf{x}}{\mathbf{r}}$ be a finite orientable relative presentation with no proper power relators. Assume that $\P$ is diagrammatically reducible and that the natural homomorphism $G \stackrel{\mathrm{nat}}{\ra} G(\P)$ admits a left inverse in the category of groups. Then the following are equivalent:
\begin{itemize}
\item[(a)] $G \stackrel{\mathrm{nat}}{\ra} G(\P)$ is an isomorphism;
\item[(b)] the image of $G \stackrel{\mathrm{nat}}{\ra} G(\P)$ has finite index in $G(\P)$;
\item[(c)] the cellular model $L(\P)$ collapses to its subcomplex $K(G,1)$: $L(\P) \searrow K(G,1)$.
\end{itemize}
Thus, if $L(\P)$ does not collapse to $K(G,1)$, then $G$ has infinite index in $G(\P)$.
\end{thm}

\begin{proof} Only the statement (b) $\Rightarrow$ (c) requires proof. The fact that $\P$ is orientable and diagrammatically reducible means that every nonempty spherical picture over $\P$ contains a dipole. Since there are no proper power relators Theorem \ref{thm:DRAspherical}(b) and Lemma \ref{Lemma:KG1} imply that the cellular model $$L = L(\P) = K(G,1) \vee \bigvee_\mathbf{x} S^1_x \cup \bigcup_\mathbf{r} c^2_r$$ is aspherical. The fact that $G \stackrel{\mathrm{nat}}{\ra} G(\P)$ has a left inverse $\nu: G(\P) \ra G$ implies that there is regular $G$-covering $p:\overline{L} \ra L$ for which the quotient map $q: \overline{L} \ra X$ that collapses $p^{-1}(K(G,1)) = \widetilde{K(G,1)} \simeq \ast$ to a point is a homotopy equivalence. As in \cite[Theorem 2.3]{BShift}, this implies that the two-complex $X$ is the cellular model of a $G$-symmetric presentation for the kernel of $\nu$. Arguing as in the proof of \cite[Lemma 3.1]{GH95}, any spherical picture (or diagram) $\mathcal{K}$ for $X$ gives rise to a relative picture $\mathcal{L}$ over $\P$ and the presence of a dipole in $\mathcal{L}$ implies that the ordinary picture $\mathcal{K}$ also has a dipole. Now the proper power hypothesis for $\P$ implies that no two-cell of $X$ is attached by a proper power, so it follows that $X$ is diagrammatically reducible (DR) in the ordinary sense, as described by Gersten \cite[Definition 3.1]{Gersten87}. In particular $X$ is aspherical in the topological sense: $\pi_2X = 0$. Since $G(\P) \cong \pi_1X \rtimes G$, there is a bijection $\pi_1X \equiv G(\P)/G$ and so the hypothesis (b)  implies that $\pi_1X$ is finite. The fact that $\pi_2X = 0$ thus implies that $X$ is simply connected and so by the result \cite[Theorem 2.4]{CorsonTrace} of Corson and Trace, $X$ is a collapsible two-complex. A sequence of elementary  collapses that reduces $X$ to a point determines a sequence of elementary collapses that reduces $L(\P)$ to its subcomplex $K(G,1)$.
\end{proof}

\begin{ex}[Diagrammatically reducible versus finiteness]\label{ex:DRinf}
(a) The natural homomorphism associated to the relative presentation $\mathcal{J} =$ \linebreak $\pres{\pres{g}{g^4},x}{x^4gx^{-3}g^2}$ of Example \ref{ex:finiteorderx} is an isomorphism of cyclic groups of order four, so $\mathcal{J}$ is aspherical by Lemma \ref{Lemma:G=H}. However the cellular model $L(\mathcal{J})$ does not collapse to the subcomplex $K(\pres{g}{g^4},1)$, so the relative presentation $\mathcal{J}$ is aspherical, yet it is not diagrammatically reducible, by Theorem \ref{thm:DRinf}. This means that $\mathcal{J}$ admits nonempty reduced spherical pictures. To our knowledge, no such pictures have yet been constructed.

(b) Theorem~\ref{thm:DRinf} implies that the diagrammatically reducible relative presentations $\mathcal{P}=\pres{G,x}{x^2gx^{-1}h}$ obtained in~\cite{ECap} define infinite groups $G(\mathcal{P})$. (More generally, it implies that if the relative presentation $\mathcal{P}=\pres{G,x}{x^{k+1}gx^{-k}h}$ ($k\geq 1$) is diagrammatically reducible then $G(\mathcal{P})$ is infinite.) As we now describe, it follows that the (almost complete) classification of the finite cyclically presented groups $H(n,m)$ and $G_n(m,k)$ of~\cite{GH95},\cite{WilliamsCHR} can be obtained without the need of appealing to~\cite{Odoni},\cite{WilliamsUnimodular}. The groups $H(n,m)=G_n(m,1)$ and $G_n(m,k)$ each have $\Z_n$ extensions defined by relative presentations of the form $\mathcal{P}=\pres{\pres{g}{g^n}}{x^2g^{m-k}x^{-1}g^k}$. The articles~\cite{GH95},\cite{WilliamsCHR} use the (almost complete) classification of diagrammatically reducible relative presentations $\pres{G,x}{x^2gx^{-1}h}$ (see Theorem~\ref{thm:l=2k=-1}) to obtain corresponding (almost complete) classifications of the topologically aspherical cyclic presentations. Since groups defined by topologically aspherical presentations are torsion-free, it follows that the corresponding cyclically presented groups are either trivial or infinite. The articles~\cite{Odoni},\cite{WilliamsUnimodular} use techniques from algebraic number theory and circulant matrices to classify the perfect groups $H(n,m)$ and $G_n(m,k)$. It transpires that, in the cases where the cyclic presentation is shown to be aspherical (which is precisely when the corresponding relative presentation $\mathcal{P}$ is known to be diagrammatically reducible), the group it defines is not perfect, and hence is non-trivial, and so is infinite. Theorem~\ref{thm:DRinf} allows us to sidestep the classifications of the perfect groups by concluding that in the diagrammatically reducible cases the group $G(\mathcal{P})$, and hence the groups $H(n,m)$ and $G_n(m,k)$, are infinite.
\end{ex}

A number of `asphericity' classifications in the literature are actually classifications of diagrammatic reducibility \cite{BP},\cite{ECap},\cite{HM},\cite{AE14},\cite{AEJ17},\cite{AAE14},\cite{EdjJuh14}. Some of these papers implicitly rely on \cite[Lemma 3]{BBP}, which is stated and proved only for the concept of (weak) asphericity. As a result there is a slight gap in the literature due to the distinction between asphericity and diagrammatic reducibility. The following enhanced version of \cite[Lemma 3]{BBP} closes this gap.

\begin{lemma}\label{lem:freesbgpeltsimplyDR} Let $\P = \pres{G,x}{r}$ be a one-relator relative presentation where the relator $r$ is cyclically reduced in the free product $G \ast F$ and is not an element of $G$ (but may be a proper power in $G \ast F$). If the $G$-coefficients occurring in $r$ are all contained in a free subgroup of $G$, then $\P$ is diagrammatically reducible.
\end{lemma}

\begin{proof} Let $\mathbb{P}$ be a connected spherical picture over $\P$. Since the corner labels of the picture lie in a free subgroup of $G$, we can assume that $G$ is free (similar to Lemma~\ref{Lemma:Coefficient}). As in \cite[Section 1.6]{BP}, the picture $\mathbb{P}$ \em lifts \em to a picture $\widehat{\mathbb{P}}$ over an ordinary one-relator presentation $\widehat{\P}$ for $G(\P)$. Since the relator $r \in G \ast F$ is cyclically reduced, the one-relator presentation $\widehat{\P}$ is \em staggered \em in the sense of \cite[Chapter III.9]{LyndonSchupp}. It follows from \cite[Proposition III.9.7]{LyndonSchupp} that the lifted picture $\widehat{\mathbb{P}}$ contains a dipole and moreover this dipole is supported by a pair of discs joined by an arc labeled by the generator $x$. (Note that \cite[Chapter III.9]{LyndonSchupp} permits the identification of a preferred subset of generators to support the staggered structure.) This means that the dipole in $\widehat{\mathbb{P}}$ actually detects a dipole in $\mathbb{P}$.
\end{proof}

\begin{re} The proof that we have given for Lemma~\ref{lem:freesbgpeltsimplyDR} depends on \cite[Propositions III.9.7]{LyndonSchupp}, which is proved in \cite{LyndonSchupp} by an extremely complex multiple induction. An alternative approach to these results appears in \cite[Section 4]{Howie87}, where the inductive complexities are sublimated within a conceptually transparent \em tower argument \em that ultimately rests on a structure theorem for staggered two-complexes $X$ satisfying $H^1X = 0$ \cite[Lemma 3]{Howie87} (see also \cite[Lemma 5.3]{Gersten87}). All of these methods, including those in \cite[Chapter III.9]{LyndonSchupp}, imply that any staggered two-complex with no proper power relators is DR in the sense of \cite{AJSColor},\cite{Gersten87}. Further innovations to the staggering concept were applied to equations over groups by Anshel \cite{Anshel91} and subsequently adapted to asphericity questions in \cite{Bogley91}.
\end{re}

We close this section with a brief discussion of combinatorial geometric methods that have been used by other authors in related contexts. The article \cite{Ol84} by Olshanskii includes concise historical remarks concerning geometric methods such as the van Kampen lemma, small cancellation theory, and aspherical presentations, as well as a survey of existence results in group theory obtained by the author using related geometric methods. The book \cite{Ol91} also develops a comprehensive treatment of geometric methods, including the formulation of a very general concept of asphericity for \em graded presentations. \em In \cite[Theorem C(1)]{P95}, Prishchepov adapts Olshanskii's graded concept to the setting of one-relator relative presentations in which the relator has free product length four; Prishchepov's result remains true as stated if one interprets the term `aspherical' in the sense of Definition~\ref{Definition:Aspherical}. We state this result in Theorem~\ref{thm:Prishchepov}. In \cite{Chalk98}, Chalk proves that certain cyclic presentations are `aspherical'; his proof actually shows that these (ordinary) presentations are DR in the sense of \cite{AJSColor},\cite{Gersten87}. Using \cite[Theorem 4.1]{BShift},  these results imply that certain relative presentations of the form $\pres{\pres{g}{g^n},x}{x^rg^lx^{-1}g^{-1}}$ are aspherical in the sense of Definition~\ref{Definition:Aspherical}. Chalk also discusses the relationship between asphericity for ordinary and relative presentations \cite[page 1520]{Chalk98}.

\subsection{Star graph and the weight test}\label{sec:WeightTest}

In this section we outline methods for showing that a given relative presentation $\mathcal{P}$ is diagrammatically reducible. The \em star graph $\mathcal{P}^{st}$ \em  of $\mathcal{P}$ is a graph whose vertex set is $\mathbf{x}\cup \mathbf{x}^{-1}$ and edge set is $\mathbf{r}^{*}$. For $r\in \mathbf{r}^{*}$, write $r=Sg$  where $g\in G$ and $S$ begins and ends with a member of $\mathbf{x}\cup \mathbf{x}^{-1}$. The initial and terminal functions are given as follows: $\iota (r)$ is the first symbol of $S$, and $\tau(r)$ is the inverse of the last symbol of $S$. The labelling function on the edges is defined by $\lambda(r)=g^{-1}$ and is extended to paths in the
usual way. For example let  $\mathbf{x}=\{x\}$ and $\mathbf{r}=\{x^lgx^kh\}$. If $l>0$ and $k>0$ the star graph is given by Figure~\ref{fig:stargraphs}(a) and  if $l>0$ and $k<0$ then the star graph is given by Figure~\ref{fig:stargraphs}(b). A non-empty cyclically reduced cycle (closed path) in $\mathcal{P}^{st}$ will be called \em admissible \em  if it has trivial label in $G$. Each inner region of a reduced picture over $\mathcal{P}$ supports an admissible cycle in $\mathcal{P}^{st}$.

      \begin{figure}
        \begin{center}
\begin{tabular}{l@{\qquad\qquad}l}
  \begin{pspicture}(-0.3,0.5)(6,4)
  \footnotesize
  \psset{xunit=0.7cm}
  \psset{yunit=0.7cm}
      \psdots[dotscale=1.5](0.05,3)(5.95,3)
      \rput(-0.5,3){$\bar{x}$}
      \rput(6.5,3){$x$}
      \pscurve(0,3)(3,5)(6,3)
      \pscurve(0,3)(3,4.2)(6,3)
      \pscurve(0,3)(3,3.5)(6,3)
      \pscurve(0,3)(3,1)(6,3)
      \psline[arrowscale=1.2]{->}(2.9,5)(3.1,5)
      \rput(3,5.3){$g$}
      \psline[arrowscale=1.2]{->}(2.9,4.2)(3.1,4.2)
      \rput(3,4.5){$h$}
      \psline[arrowscale=1.2]{->}(2.9,3.5)(3.1,3.5)
      \rput(3,3.8){$1$}
      \psline[arrowscale=1.2]{->}(2.9,1)(3.1,1)
      \rput(3,1.3){$1$}
      \pscurve[linestyle=dotted](2,3.2)(1.8,2.4)(2,1.5)
      \rput(3.5,2.5){$(l+k)-2$~edges}
      \rput(3,0){(a)}
      \end{pspicture}

&
\begin{pspicture}(-0.3,0.5)(6,4)
  \footnotesize
  \psset{xunit=0.7cm}
  \psset{yunit=0.7cm}
      \psdots[dotscale=1.5](0.05,3)(5.95,3)
      \rput(0,2.3){$\bar{x}$}
      \rput(6,2.3){$x$}
      \pscurve(0,3)(3,5)(6,3)
      \pscurve(0,3)(3,3.3)(6,3)
      \pscurve(0,3)(3,2.8)(6,3)
      \pscurve(0,3)(3,1)(6,3)
      \psline[arrowscale=1.2]{->}(3.1,5)(2.9,5)
      \rput(3,5.3){$1$}
      \psline[arrowscale=1.2]{->}(3.1,3.3)(2.9,3.3)
      \rput(3,3.6){$1$}
      \psline[arrowscale=1.2]{->}(2.9,2.8)(3.1,2.8)
      \rput(3,2.4){$1$}
      \psline[arrowscale=1.2]{->}(2.9,1)(3.1,1)
      \rput(3,0.7){$1$}
      \pscircle(-0.5,3){0.4}
      \psline[arrowscale=1.2]{->}(-1.05,2.9)(-1.05,3.1)
      \rput(-1.3,3){$h$}
      \pscircle(6.5,3){0.4}
      \psline[arrowscale=1.2]{->}(7.05,2.9)(7.05,3.1)
      \rput(7.3,3){$g$}
      \pscurve[linestyle=dotted](2,4.4)(1.9,3.95)(2,3.5)
      \rput(3.3,4.0){$|k|-1$~edges}
      \pscurve[linestyle=dotted](4,2.6)(4.1,2.15)(4,1.7)
      \rput(2.8,1.8){$l-1$~edges}
      \rput(3,0){(b)}
      \normalsize
\end{pspicture}
\end{tabular}
\end{center}

          \caption{Star graphs.\label{fig:stargraphs}}
      \end{figure}
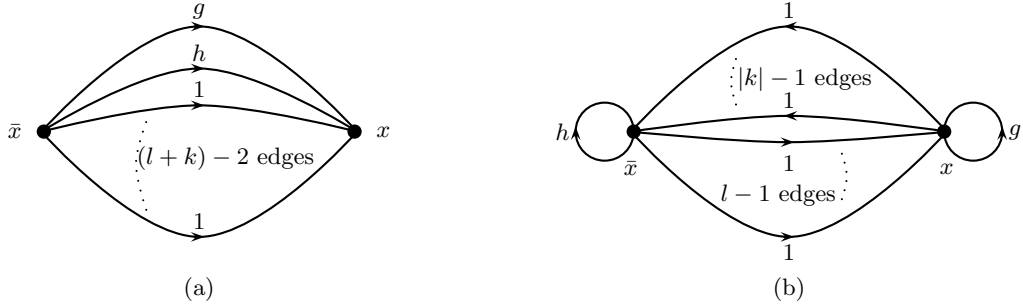

A \em weight function \em  $\theta$ is a real-valued function on the set of edges of $\mathcal{P}^{st}$ that satisfies
$\theta(Sg)=\theta(S^{-1}g^{-1})$ where $Sg=r\in\mathbf{r}^{*}$. The weight of a closed cycle is the sum of the weights of the constituent edges. A weight function is \em weakly aspherical \em  if  the following conditions are satisfied:
\begin{itemize}
  \item[(I)] 	Let $r\in \mathbf{r}^{*}$, with $r= x_{1}^{\varepsilon_{1}}g_{1}\ldots x_{n}^{\varepsilon _{n}}g_{n}$. Then
\[\sum_{i=1}^{n} (1-\theta(x_{i}^{\varepsilon_{i}}g_{i}\ldots x_{n}^{\varepsilon_{n}}g_{n}x_{1}^{\varepsilon_{1}}g_{1}\ldots x_{i-1}^{\varepsilon_{i-1}}g_{i-1}))\geq 2.\]

  \item[(II)] The weight of each admissible cycle in $\mathcal{P}^{st}$ is at least 2.
\end{itemize}

A weakly aspherical weight function on $\mathcal{P}^{st}$ is \em aspherical \em if each cyclically reduced closed cycle in $\mathcal{P}^{st}$ has non-negative weight.

\begin{thm}[{{\cite[Theorem~2.1]{BP}, \cite{GerstenUnpubl}, \cite[Theorem~2.2]{How91}}}]\label{Theorem:WeightTest}
If the star graph $\mathcal{P}^{st}$  admits a (weakly) aspherical weight function, then
        $\mathcal{P}$ is (weakly) diagrammatically reducible.
\end{thm}

\subsection{Curvature distribution}\label{sec:curredistrib}

Another method that can be used to show that a given relative presentation $\mathcal{P}$ is diagrammatically reducible is \em curvature distribution \em  (see, for example, \cite{Edj91a}). Assume that  $\mathbb{P}$  is a non-trivial reduced, connected, spherical picture
over $\mathcal{P}$. We proceed as follows. Contract $\partial \mathbb{P}$ to a point, which is then deleted. This way the regions $\Delta$ of the amended picture, also called $\mathbb{P}$, are simply connected and tesselate the 2-sphere. The region obtained from the annular region is called the \em distinguished region \em $\Delta_{0}$ and its label may or may not equal the identity in $G$; whereas the label of every other region is equal to the identity. (We remark here that when drawing it is often customary to shrink each disc/vertex of $\mathbb{P}$ to a point -- an example is shown in Figure~\ref{fig:dipoleandvertex}(b). In what follows we shall simply use the term vertex.)

The next step is to define an \em angle function, \em  that is, a real-valued function on the set of corners of $\mathbb{P}$, with the further condition that the sum of all the angles at any given vertex is equal to $2\pi$. This way each vertex has zero cuvature; and if $\Delta$ is a $k$-gonal region then the curvature $c(\Delta)$ is equal to $2\pi - \sum_{i=1}^k (\pi-\theta_{i})$  where the sum is taken over the corner angles $\theta_{i}\ (1\leq i \leq k)$ of $\Delta$. Therefore the total curvature $c(\mathbb{P})$ of $\mathbb{P}$ is given by $c(\mathbb{P})=\sum_{\Delta \in \mathbb{P}} c(\Delta)$.  Given this, it is a consequence of Euler's formula, for example, that $c(\mathbb{P}) = 4 \pi$ and so $\mathbb{P}$ must contain regions of positive curvature.

For example, note that each corner $c$ of $\mathbb{P}$ corresponds to an edge $e$ in the star graph $\mathcal{P}^{st}$ and so if $\theta$ is a weight function on $\mathcal{P}^{st}$, then one possible angle function assigns $\pi \cdot \theta(e)$ to the corner $c$ of $\mathbb{P}$. Let $\hat{\mathbb{P}}$ be the dual of $\mathbb{P}$ with angles induced by those of $\mathbb{P}$. Then the condition (I) for the weight function $\theta$ in Section \ref{sec:WeightTest} enforces that each region of $\hat{\mathbb{P}}$ has non-positive curvature, while the condition (II) for $\theta$ implies that each vertex of $\hat{\mathbb{P}}$ has non-positive curvature since $\mathbb{P}$ is assumed to be reduced. This leads to a contradiction to Euler's formula (see e.g.~\cite[Sect.~1.1]{BPpi2gen}), showing that the relative presentation $\P$ is diagrammatically reducible.

Note that by assigning angles according to a weight function on the star graph, the angles are assigned in a uniform way that is not sensitive to local configurations in the given picture $\mathbb{P}$. We now describe a method, first introduced in this context in \cite{ECap}, that ultimately results in a more flexible assignment of angles.

If we believe that $\mathcal{P}$ is diagrammatically reducible and so no reduced spherical picture $\mathbb{P}$ can exist, then we try to obtain a contradiction by showing that the total curvature is less than $4\pi$. The strategy is to show that the positive curvature that may exist in $\mathbb{P}$ can be sufficiently compensated by negative curvature. To this end we systematically locate each  $\Delta \neq \Delta_{0}$ satisfying $c(\Delta)>0$ and distribute $c(\Delta)$ to near regions $\hat{\Delta}$ of $\Delta$. For such regions $\hat{\Delta}$ define $c^{*}(\hat{\Delta})$ to equal $c(\hat{\Delta})$ plus all the positive curvature $\hat{\Delta}$ receives minus all the curvature $\hat{\Delta}$ distributes as a result of this \em distribution scheme. \em It is clear that $c(\mathbb{P})$ is at most $\sum_{\Delta \in \mathbb{P}} c^{*}(\hat{\Delta})$. The final steps are to show that $c^{*}(\hat{\Delta}) \leq 0$ for $\hat{\Delta} \neq \Delta_{0}$; and that $c^{*}(\hat{\Delta}_{0})<4\pi$.

Let $v$ be a vertex of $\mathbb{P}$. A standard angle function is to assign  the angle 0 to each corner of $v$ which forms part of a region of degree two and assigns the angles ${2 \pi}/{d(v)}$ to the remaining $d(v)$ corners.  Therefore if $\Delta$ is a $k$-gonal region ($k > 2$) with vertices $v_1, \ldots, v_k$ such that $d(v_i) = d_i$ ($1 \leq i \leq k$) then $c(\Delta)=c(d_1,\ldots,d_k) = (2-k) \pi + \sum_{i=1}^k {2 \pi}/{d_i}$; or if $d(\Delta)=2$ then
$c(\Delta)=0$.  Suppose that $d(v)>2$ for each vertex $v$. Since $c(3,3,3,3,3,3)=0$ we have, for example, that if $c(\Delta)>0$ then $3 \leq d(\Delta) \leq 5$. Since the label of each
region corresponds to an admissible closed path in the star graph, if we consider Figure~\ref{fig:stargraphs}(a)
then only degree four is possible and if both $g$ and $h$ are involved this forces the label
to be (up to cyclic permutation and inversion) one of $gh^{-1}11^{-1}$, $gh^{-1}g1^{-1}$, $hg^{-1}h1^{-1}$ or $gh^{-1}gh^{-1}$ and then either $g=h$,
$h=g^{2}$, $g=h^{2}$ or $(gh^{-1})^{2} = 1$ in $G$. A similar analysis of Figure~\ref{fig:stargraphs}(b)
yields $g=h^{\pm 1}$, $g=h^{\pm 2}$ or  $h=g^{\pm 2}$ in $G$.

\section{One relator relative presentations with length four relator}\label{section:LengthFour}

Asphericity of one-relator relative presentations has been considered in the articles
\begin{itemize}
  \item[\ ] \cite{BP},\cite{BBP},\cite{ECap},\cite{HM},\cite{Davidson09},\cite{AE14},\cite{AEJ17},\cite{Ahmad},\cite{AAE14},\cite{P95},\cite{EdjJuh14}\eqnum\label{eq:articles}
\end{itemize}
and in~\cite{Met},\cite{SKK},\cite{SKK6}. Those listed at~(\ref{eq:articles}) contain classifications of aspherical one-relator relative presentations when the relator has free product length four; that is for relative presentations %
\[ \mathcal{Q}=\pres{G,x}{x^lgx^kh}\]
for certain values of the parameters $l>0$, $k\neq 0$ and where $g,h$ are non-trivial elements of $G$. Specifically, they
provide detailed asphericity studies for the relative presentations $\mathcal{Q}$ for the cases $\{l,k\}=\{2,1\},\{3,1\},\{4,1\},\{3,2\},\{n,1\}$\ ($n\geq 2$), $\{l,k\}$ ($l,k>0$), $\{2,-1\}$, $\{3,-1\}$, $\{n,-1\}$ ($n\geq 4$), $l>0,k<0$ and $l>2|k|$ or $|k|>2l$.  Our purpose in this section is to survey the results from these articles, bring them up to date (where appropriate) and present them under the definition of asphericity given in Definition~\ref{Definition:Aspherical}. We express some of our results in terms of the parameter $$\mu=1/|g|+1/|h|+1/|gh^{-1}|.$$

\subsection{The Platonic case}\label{sec:platonic}

We introduce the following condition:

\begin{description}[leftmargin=!,labelwidth=\widthof{($K_6^-$)}]
  \item[($P$)] $\mu>1$ and $g\neq h$.
\end{description}

For many values of $l,k$ it has been shown that, when ($P$) holds, spherical pictures based on the Platonic solids can be constructed, and hence $\mathcal{Q}$ is non-diagrammatically reducible. Further, in some cases these have been used to prove that $\mathcal{Q}$ is non-weakly aspherical.

Specifically, we have the following results. Suppose that ($P$) holds. If $\{l,k\} = \{2,1\}, \{3,2\}, \{4,1\}, \{n,1\}\ (n\geq 2), \{2,-1\}, \{3,-1\}$, or if $l>0,k>0$ then $\mathcal{Q}$ is non-diagrammatically reducible by~\cite[Theorem~3.4]{BP}, \cite[Lemmas~10 and~12]{HM}, \cite[Lemma~3.4]{AE14}, \cite[Lemma~4.3]{ECap}, \cite[Lemma~3.1]{AAE14}, or~\cite[Lemma~3.5(ii)]{AEJ17} respectively. If $\{l,k\}=\{3,1\}$ or $\{3,-1\}$ then $\mathcal{Q}$ is non-weakly aspherical by~\cite[Lemma~7]{BBP} or \cite[pages~71--75]{Ahmad}, respectively. If $l\geq 2, k=-1$ and $|g|,|h|\geq 3$ and ($P$) holds then $\mathcal{Q}$ is non-weakly aspherical by~\cite[Lemmas~3.15, 3.16, 3.17]{Davidson09}. We expect that these results hold more generally and may be extended to prove non-asphericity. We therefore make the following conjecture:

\begin{conj}\label{conj:platonics}
If Case ($P$) holds then $\mathcal{Q}$ is non-aspherical.
\end{conj}

\subsection{Euclidean curvature and short admissible closed path cases}\label{sec:ListOfCases}
It is convenient to introduce the (possibly overlapping) cases
\begin{itemize}
  \item[] ($P$), ($Z$), (${M}$), ($J_4$), ($J_6$), ($K_5$), ($K_6^+$), ($K_6^-$), ($L_6$) \eqnum\label{eq:ListOfCases}
\end{itemize}
where ($P$) is defined in Section~\ref{sec:platonic} and the remaining cases are defined as follows:

\begin{description}[leftmargin=!,labelwidth=\widthof{($K_6^-$)}]
  \item[($Z$)] $g=h$ and $|g|<\infty$;
  \item[(${M}$)] $g=h^{-1}$ and $|g|<\infty$;
  \item[($J_4$)] ($g=h^2$, $|h|=4$) or ($h=g^2$, $|g|=4$);
  \item[($J_6$)] ($|g|=2$, $|h|=3$, $[g,h]=1$) or ($|g|=3$, $|h|=2$, $[h,g]=1$);
  \item[($K_5$)] ($g=h^2$, $|h|=5$) or ($h=g^2$, $|g|=5$);
  \item[($K_6^+$)] ($g=h^2$, $|h|=6$) or ($h=g^2$, $|g|=6$);
  \item[($K_6^-$)]  ($g=h^{-2}$, $|h|=6$) or ($h=g^{-2}$, $|g|=6$);
  \item[($L_6$)] ($g=h^3$, $|h|=6$) or ($h=g^3$, $|g|=6$).
\end{description}

In Section~\ref{section:Methods} we saw that, amongst others, the presence of relations
\begin{equation}
 g=h^{\pm 1}~\mathrm{or}~g=h^{\pm 2}~\mathrm{or}~h=g^{\pm 2} \label{eq:shortcycles}
\end{equation}
lead to short admissible closed paths in the star graph and so present obstacles in proving asphericity for $\mathcal{Q}$. In~\cite{BW1} it was suggested that the Euclidean curvature condition $\mu=1$
may sometimes also present a barrier to asphericity. The full picture is not yet clear, but it seems that the
condition~(\ref{eq:shortcycles}) interacts with the condition $\mu=1$ to present a liminal space where some presentations $\mathcal{Q}$ are aspherical and others are non-aspherical. While we believe that the cases listed at~(\ref{eq:ListOfCases}) capture many situations where non-aspherical presentations may occur, they are not (at least for $k<0$) the only ones, as we shall see in Examples~\ref{ex:333} and~\ref{ex:infinitewithtorsion}.

The cases listed at~(\ref{eq:ListOfCases}) satisfy either (\ref{eq:shortcycles}) or $\mu\geq 1$ or both, in Case~($P$) the group $\gpres{g,h}$ is a quotient of a spherical triangle group and in the remaining cases it is cyclic. Specifically, Cases~($P$) and ($Z$) satisfy $\mu>1$ and Cases ($Z$), (${M}$), ($J_4$), ($K_6^+$), ($K_6^-$) and ($K_5$) all satisfy~(\ref{eq:shortcycles}). The case ($K_5$) simultaneously satisfies it in two ways: $g=h^{\pm 2}$ and $h=g^{\mp 2}$. Cases ($J_4$), ($J_6$), ($K_6^-$), ($L_6$) satisfy $\mu=1$ with $(\{|g|,|h|\},|gh^{-1}|) = (\{2,4\},4)$, $(\{2,3\},6)$, $(\{3,6\},2)$, $(\{2,6\},3)$, respectively. The remaining values of $(\{|g|,|h|\},|gh^{-1}|)$ that yield $\mu=1$ are $(\{3,3\},3)$, $(\{4,4\},2)$. Examples~\ref{ex:333}(a),(b), below, show that $\mathcal{Q}$ may be non-aspherical
when $\gpres{g,h}$ is non-cyclic and $(\{|g|,|h|\},|gh^{-1}|)=(\{2,3\},6)$ or $(\{3,3\},3)$. We are not aware of any non-aspherical example in the case $(\{4,4\},2)$.

\begin{ex}[Finite groups outside the cases~(\ref{eq:ListOfCases})]\label{ex:333}
(a) Let $G$ be the group $G=\pres{g,h}{g^2,h^3,(gh)^2(g^{-1}h^{-1})^2}\cong S_3\oplus \Z_3$ and let $\mathcal{Q}=\pres{G,x}{x^2gx^{-1}h}$. Here we have $(|g|,|h|,|gh^{-1}|)=(2,3,6)$, and so $\mu=1$. Using GAP~\cite{GAP} we have that $G(\mathcal{Q})$ is a finite, solvable, group of order $27216=2^4\cdot 3^5\cdot 7$ and derived length~6. By Theorem~\ref{Theorem:Theory}(c) the presentation $\mathcal{Q}$ is non-aspherical.

(b) Let $G=\pres{g,h}{g^3,h^3,ghg^{-1}h^{-1}}\cong \Z_3\oplus \Z_3$ and let $\mathcal{Q}=\pres{G,x}{x^2gx^{-1}h}$. Here we have $(|g|,|h|,|gh^{-1}|)=(3,3,3)$, and so $\mu=1$. We have that $G(\mathcal{Q})$ is a finite, solvable, group of order $13608=2^3\cdot 3^5\cdot 7$ and derived length~5 and hence $\mathcal{Q}$ is non-aspherical.

(c) Let $G=\pres{g}{g^8}\cong \Z_8$ and let $\mathcal{Q}=\pres{G,x}{x^2g^2x^{-1}g}$.  As was essentially first observed in~\cite[Table~4]{BV}, the group $G(\mathcal{Q})$ is a finite, solvable, group of order $2361960=2^3\cdot 3^{10}\cdot 5$, and derived length~4. Hence $\mathcal{Q}$ is non-aspherical. As we shall see in Corollary~\ref{cor:l=2k=-1}, this settles a previously open case of~\cite[Theorem~1.1]{ECap}.
\end{ex}

We may also prove non-asphericity of $\mathcal{Q}$ by exhibiting an element of prohibited finite order.

\begin{ex}[Infinite groups with torsion outside the cases~(\ref{eq:ListOfCases})]\label{ex:infinitewithtorsion}

(a) Let $G=\pres{g}{g^n}$ where $n\geq 9$ is odd and let $\mathcal{Q}=\pres{G,x}{x^2g^2x^{-1}g^{-1}}$.
As in Example~\ref{ex:3mfdgps} we have that $G(\mathcal{Q})\cong F(2,n)\rtimes \Z_n$ where
\[F(2,n)=\pres{x_0,\ldots ,x_{n-1}}{x_ix_{i+1}=x_{i+2}\ (0\leq i\leq n-1)}\]
is a Fibonacci group. Setting $w=x_0x_1\ldots x_{n-1}$, we have that $w^2=1$ in $F(2,n)$ (see~\cite[page~83]{JBook2}), and by (the proof of)~\cite[Proposition~3.1]{BV} $w$ is an element of order exactly two. Thus $G(\mathcal{Q})$ contains an element of order~2 so by Theorem~\ref{Theorem:Theory}(c) we have that $\mathcal{Q}$ is non-aspherical.

(b) Let $G=\pres{g}{g^9}$ and let $\mathcal{Q}=\pres{G,x}{x^2g^2x^{-1}g}$.  Then we have that $G(\mathcal{Q})\cong H(9,3)\rtimes \Z_9$ where
\[H=H(9,3)=\pres{x_0,\ldots ,x_{8}}{x_ix_{i+3}=x_{i+1}\ (0\leq i\leq 8)}\]
is a \em Gilbert-Howie group \em which, by~\cite[Lemma~15]{COS}, is infinite. Setting $$w=x_0x_4x_8x_3x_7x_2x_6x_1x_5,$$ the diagram~\cite[Figure~2]{GW12a} (which was derived from ~\cite[Figure~4.1(i)]{ECap}) can be used to show that $w^2=1$ in $H$. Moreover the group $H/\ngpres{w}^H$ is finite of order $2^{15}\cdot 7$. Thus $w$ has order exactly two in $H$ so $G(\mathcal{Q})$ contains an element of order~2 and hence $\mathcal{Q}$ is non-aspherical, by Theorem~\ref{Theorem:Theory}(c). (This order two element $w$ was first obtained in the proof of~\cite[Lemma~4]{HW16}.)

(c) Let $G=\pres{g}{g^7}$ and let $\mathcal{Q}=\pres{G,x}{x^2g^2x^{-1}g}$. Then we have that $G(\mathcal{Q})\cong H(7,3)\rtimes \Z_7$ where
\[H=H(7,3)=\pres{x_0,\ldots ,x_{6}}{x_ix_{i+3}=x_{i+1}\ (0\leq i\leq 6)}\]
is a Gilbert-Howie group which, by~\cite[Theorem~3.3]{GH95} (due to Thomas), is infinite. (Alternatively use GAP to show that the second derived subgroup of $H$ is free abelian of rank~8). Setting $w=x_0x_3x_6x_2x_5x_1x_4$, we have that~\cite[Figure~1]{GW12a} (see also~\cite[Figure~4.1(h)]{ECap}) can be used to show that $w^2=1$ in $H$. Moreover the group $H/\ngpres{w}^H$ is finite of order $128$. Thus $w$ has order exactly~2 in $H$ so $G(\mathcal{Q})$ contains an element of order~2 and hence $\mathcal{Q}$ is non-aspherical by Theorem~\ref{Theorem:Theory}(c).
\end{ex}

In Cases ($J_4$), ($J_6$), Theorem~A of~\cite{BW1} shows that the relative presentation $\mathcal{Q}$ is aspherical only when the natural homomorphism $G \ra G(\P)$ is an isomorphism.

\begin{thm}[{{\cite[Theorem~A]{BW1}}}]\label{thm:AsphericityForJ4J6}
Suppose that Case ($J_n$) holds ($n=4$ or $6$). Then the following are equivalent:
\begin{itemize}
  \item[(a)] $\mathcal{Q}$ is aspherical;
  \item[(b)] $|l+k|=1$ and either $l\equiv 0$~mod~$n$ or $k\equiv 0$~mod~$n$;
  \item[(c)] The natural homomorphism $G \ra G(\P)$ is an isomorphism.
\end{itemize}
\end{thm}

We now prove the corresponding result for Cases ($Z$) and (${M}$):

\begin{thm}\label{thm:ZM} If the Case ($Z$) or (${M}$) holds, then the following are equivalent:

\begin{enumerate}
\item[(a)] $\mathcal{Q}$ is aspherical;
\item[(b)] $|l+k| = 1$ and $l=0$ or $k= 0$;
\item[(c)] The natural homomorphism $G \ra G(\P)$ is an isomorphism.
\end{enumerate}
\end{thm}

\begin{proof}
That (c) $\Rightarrow$ (a) follows from Lemma~\ref{Lemma:G=H}. We show (b) $\iff$ (c). We have $G(\P) \cong G \ast_{\sgp{g}} G_0$ where $G_0$ has an ordinary presentation $G_0 \cong \pres{g,x}{g^n, x^lgx^kg^\epsilon}$ for some $n\geq 2$ and where $\epsilon = \pm1$. Thus $G \ra G(\P)$ is an isomorphism if and only if $G_0 = \sgp{g} \cong \Z_n$. When $\epsilon = -1$, \cite[Theorem~A(c)]{BW2} shows that $G_0 = \sgp{g} \iff$ (b) holds. Likewise, when $\epsilon = 1$, \cite[Theorem~3.2.2.2]{McD17} shows that $G_0 = \sgp{g} \iff$ (b) holds. (These two theorems classify finiteness of groups of the form $G_0 = \pres{g,x}{g^n,x^lgx^kg^\epsilon}$.) Thus (b) $\iff$ (c).

We now prove that (a) $\Rightarrow$ (c). Given that $\mathcal{Q}$ is aspherical, it suffices to prove that $G_0 = \sgp{g}$. By Lemma~\ref{Lemma:Coefficient}, the relative presentation $\mathcal{Q}_0 = \pres{\sgp{g},x}{x^lgx^kg^\epsilon}$ is aspherical and so by Theorem~\ref{Theorem:Theory}(c), it suffices to prove that $G_0 = G(\mathcal{Q}_0)$ is finite. We show that $l+k \neq 0$. Supposing otherwise, if $l=-k$, then Lemma~\ref{lemma:relatorconditions}(a) implies that $l \neq 0$ (and $g \neq 1$). We have $x^lgx^{-l}g^\epsilon = 1$ in $G(\mathcal{Q})$  so $g^{-\epsilon} \in \sgp{g} \cap x^l\sgp{g}x^{-l}$ and Theorem~\ref{Theorem:Theory}(c) implies that $x^l \in \sgp{g}$ and so $x^l$ has finite order in $G_0$, contradicting the fact that if $l = -k$, then the quotient group $G_0/\<\<g\>\>$, which is generated by $x$, is infinite cyclic.

Now suppose that $\epsilon = 1$, $l+k \neq 0$, and $\mathcal{Q}$ is aspherical. In this case, the element $x^{l-k}$ is central in $G_0 \cong \pres{g,x}{g^n, x^lgx^kg }$ by \cite[Lemma~3.2.2.2]{McD17} so asphericity of $\mathcal{Q}_0$ implies that $x^{l-k} \in \sgp{g}$ by Theorem~\ref{Theorem:Theory}(c). Now $l-k \neq 0$ because otherwise the relator $x^lgx^kg = (x^lg)^2$ is a proper power, which contradicts asphericity by Lemma~\ref{lemma:relatorconditions}(c). Thus the element $x \in G_0$ has finite order, so Theorem \ref{Theorem:Theory}(c) implies that $x$ is conjugate in $G_0$ to an element of $\sgp{g}$. Consider the central quotient

\begin{eqnarray*}
\Delta &=& G_0/\sgp{x^{l-k}} \cong \pres{g,x}{g^n, x^{l-k}, (x^lg)^2}\\
&\cong& \pres{g,x,u}{g^n, x^{l-k}, (x^lg)^2,u=x^l}\\
&\cong& \pres{g,x,u}{g^n, x^{l-k}, u^{(l-k)/(l,k)}, (ug)^2,u=x^l}\\
&\cong& \pres{g,u}{g^n, u^{(l-k)/(l,k)}, (ug)^2} \ast_{u=x^l} \pres{x}{x^{l-k}}.
\end{eqnarray*}
The fact that $x$ is conjugate to an element of $\sgp{g}$ implies that this amalgamated free product decomposition is degenerate, that is $\sgp{x} = \sgp{u} = \sgp{x^l}$, and so it follows that $(l,k) = 1$. Thus $\Delta \cong \pres{g,x}{g^n, x^{l-k}, (xg)^2}$ is the ordinary triangle group. The five-term homology sequence for the central extension is
$$
H_2G_0 \ra H_2\Delta \ra \sgp{x^{l-k}} \ra H_1G_0 \ra H_1\Delta \ra 0.
$$
Since $l+k \neq 0$, the $2\times 2$ relation matrix for $G_0 \cong \pres{g,x}{g^n, x^lgx^kg }$ has nonzero determinant and thus $H_2G_0 = 0$.  Since $x$ has finite order in $G_0$, this implies that $H_2\Delta$ is finite, which in turn implies that the triangle group $\Delta$ is finite and so $G_0$ is finite, as required.

It remains to consider the case where $\epsilon = -1$, $l+k \neq 0$, and $\mathcal{Q}$ is aspherical. We claim that $x$ has finite order in $G_0$. To see this, arguing as in the proof of \cite[Lemma 14]{BW2}, the relations $gx^kg^{-1} = x^{-l}$ and $g^n = 1$ imply that $x^{k^n-(-l)^n} = 1$, so $x$ has finite order unless $k^n-(-l)^n = 0$. If $k^n - (-l)^n = 0$, then since $k \neq -l$, it follows that $k = l \neq 0$ (and $n$ is even) and so we have $gx^lg^{-1} = x^{-l}$, whence $g^2$ commutes with $x^l$ and in particular $g^2 \in \sgp{g} \cap x^l\sgp{g}x^{-l}$. We cannot have $g^2 = 1$, for otherwise our relator is a proper power $x^lgx^lg^{-1} = (x^lg)^2$, contrary to Lemma~\ref{lemma:relatorconditions}(c), so Theorem~\ref{Theorem:Theory}(c) implies that $x^l \in \sgp{g}$ and so $x$ has finite order. Arguing as in the case $\epsilon = 1$, the fact that $x$ has finite order in $G_0$ implies that $(l,k) = 1$, and so \cite[Lemma~14]{BW2} shows that the group $G_0 \cong \pres{g,x}{g^n, x^lgx^kg^{-1}}$ is a semi-direct product $G_0 \cong \sgp{x} \rtimes \sgp{g}$. Thus $G_0$ is finite, as required.
\end{proof}

The remaining cases from~(\ref{eq:ListOfCases}) are ($K_5$), ($K_6^+$), ($K_6^-$), ($L_6$). These appear to be challenging cases and only partial information is known. We now survey the state of knowledge for the values of $\{l,k\}$ considered in the studies listed at (\ref{eq:articles}). In some articles non-diagrammatic reducibility was proven by explicitly constructing spheres though these, in themselves, do not prove non-asphericity. However if the corresponding group $G(\mathcal{Q})$ can be shown to be finite of order greater than six (which can sometimes be proved using GAP) an application of Theorem~\ref{Theorem:Theory}(c) proves non-asphericity.

Table~\ref{tab:Ocaseslk>0} summarizes what is known for each case. When $G(\mathcal{Q})$ is known to be finite of order greater than six the cell either contains an expression $[M,N]$, which means that $G(\mathcal{Q})$ is the $N$'th group of order $M$ in the Small Groups library~\cite{SmallGroupsLibrary}, or it contains a number, which means that $G(\mathcal{Q})$ is finite of that order, but does not have an entry in the Small Groups library. In these cases $\mathcal{Q}$ is non-aspherical by Theorem~\ref{Theorem:Theory}(c). The group of order $24530688=2^8\cdot 3^4\cdot 7 \cdot 13^2$ shows that $\mathcal{Q}$ is non-aspherical in the previously unresolved Case (E2) of~\cite{BBP}, and was studied in detail in~\cite[Lemma~9.7]{BogleyParker}.
In the case ($K_6$) with $l=2,k=-1$ the group $G(\mathcal{Q})$ is the infinite virtual three-manifold group from Example~\ref{ex:3mfdgps}(a), so $\mathcal{Q}$ is non-aspherical by Theorem~\ref{thm:3mfdcriterion}. A question mark in a cell indicates that it remains unknown if the presentation $\mathcal{Q}$ is aspherical; these cases are either isolated as unresolved exceptional cases or are the subject of questions on their asphericity statuses in the references given in the final column. We note that aspherical presentations $\mathcal{Q}$ exist in the case ($K_5$), as we shall see in Theorem~\ref{thm:K_5(l,-1)}.

\begin{table}[ht]
\begin{center}
\begin{tabular}{|l|l|l|l|l|l|}
\hline
$\{l,k\}$ & ($K_5$) & ($K_6^+$) & ($K_6^-$) & ($L_6$) & \textbf{Reference}\\\hline
$\{2,1\}$ & [165,1] & [378,8] & [342,7] & [342,7] & \cite{BP}\\ \hline
$\{3,1\}$ & [1100,26] &  ? & ? & 24530688& \cite{BBP}\\\hline
$\{4,1\}$ & [3775,3] & ?  & ? & ?& \cite{HM}\\\hline
$\{3,2\}$ & [2525,3] & ?  & ? & ?& \cite{HM}\\\hline
$\{n,1\}$ ($n\geq 5$)& ? & ?  & ? & ?& \cite{AE14}\\\hline
$l,k>0$,  & ? & ?  & ? & ? &\cite{AEJ17}\\
$\{l,k\}$ not above &  & & & & \\\hline
$\{2,-1\}$ & [55,1]                          & [336,210]              & Ex.~\ref{ex:3mfdgps}(a) & [54,6] &\cite{ECap}\\\hline
$\{3,-1\}$ & [110,2] & ?  & ? & 9072&\cite{AAE14}\\\hline
\end{tabular}
\end{center}
\caption{Resolved and unresolved Cases ($K_5$),($K_6^+$),($K_6^-$),($L_6$) in the studies~(\ref{eq:articles}).\label{tab:Ocaseslk>0}}
\end{table}

\subsection{Asphericity results}\label{sec:asphericitystudies}

We now survey results from the articles listed at~(\ref{eq:articles}) concerning asphericity and diagrammatic reducibility for relative presentations $\mathcal{Q} = \pres{G,x}{x^lgx^kh}$. In many cases we shall assume that none of~(\ref{eq:ListOfCases}) hold, those cases having been discussed in Sections~\ref{sec:platonic} and~\ref{sec:ListOfCases}. For example, the results discussed in those sections imply that if $\{l,k\} = \{2,1\}$ or $\{2,-1\}$ and~(\ref{eq:ListOfCases}) holds, then $\mathcal{Q}$ is not aspherical. See Theorem~\ref{thm:lk=21} and Corollary~\ref{cor:l=2k=-1} below. We first consider the cases where $l = \pm k$.

\begin{thm}[{{\cite[Lemma~3.3]{AEJ17}}}]\label{thm:l=k}
If $l=k\geq 1$ then $\mathcal{Q}$ is diagrammatically reducible if and only if $g = h$ or $|g^{-1}h|=\infty$.
\end{thm}

\begin{co}\label{cor:l=k}
If $l=k\neq 0$, then $\mathcal{Q}$ is aspherical if and only if $|g^{-1}h|=\infty$.
\end{co}

\begin{proof}
If $g=h$ then the relator $x^lgx^kh$ is a proper power $(x^lg)^2$ so $\mathcal{Q}$ is non-aspherical,
by Lemma~\ref{lemma:relatorconditions}(c). Thus we may assume $g\neq h$. If $|g^{-1}h|=\infty$ then $\mathcal{Q}$ is diagrammatically reducible and orientable, hence aspherical. Thus we may assume that $|g^{-1}h|<\infty$. The relator $x^lgx^lh$ implies that $(x^lg)^2 = (g^{-1}h)^{-1}$ so $x^lg$ has finite order in $G(\mathcal{Q})$. If $\mathcal{Q}$ is aspherical then Theorem~\ref{Theorem:Theory}(c) implies that $x^lg$ is conjugate to an element of $G$, i.e.\ $x^lg=whw^{-1}$ for some $h\in G$ and $w\in G*\gpres{x}$. Thus  $x^lgwh^{-1}w^{-1}$ has exponent sum $l$ in $x$ and is equal to the identity of $G(\mathcal{Q})$. But all words that are trivial in $G(\mathcal{Q})$ are products of conjugates of elements of $G$ or of the relator and so their exponent sum is a multiple of $l+k=2l$, a contradiction.
\end{proof}

\begin{thm} If $l = -k \neq 0$, then $\mathcal{Q}$  is aspherical if and only if $|g| = |h|=\infty$.
\end{thm}

\begin{proof} If $|g| = |h| = \infty$ then $\mathcal{Q}$ is orientable and an easy application of the weight test shows that $\mathcal{Q}$ is diagrammatically reducible, hence aspherical. Conversely, if $\mathcal{Q}$ is aspherical, then $1 \neq h^{-1} = x^{-k}gx^k \in G \cap x^{-k}Gx^k$ and $1 \neq g^{-1} = x^khx^{-k} \in G \cap x^kGx^{-k}$ and so Theorem~\ref{Theorem:Theory}(c) implies that $|g| = |h| = \infty$ because the elements $x^{\pm k}$ both lie outside the normal closure of $G$ in $G(\mathcal{Q})$ and so neither of the elements $x^{\pm k}$ is conjugate in $G(\mathcal{Q})$ to any element of $G$ -- see Theorem~\ref{Theorem:Theory}(c).
\end{proof}

From now on we assume that $l \neq \pm k$, in which case the relative presentation $\mathcal{Q} = \pres{G,x}{x^lgx^kh}$ is orientable with non-proper power relator. Thus if $\mathcal{Q}$ is diagrammatically reducible, then it is aspherical. If $G$ is torsion-free then $\mathcal{Q}$ is diagrammatically reducible by~\cite[Lemma~2]{SKK6}. The difficulty in proving diagrammatic reducibility comes about when $G$ has torsion. We shall use the preceding observation freely to deduce asphericity when a given result asserts diagrammatic reducibility. We start with the case $\{l,k\}=\{2,1\}$.

\begin{thm}[\cite{BP}]\label{thm:lk=21}
Let $\{l,k\}=\{2,1\}$. Then $\mathcal{Q}$ is diagrammatically reducible and hence aspherical if and only if none of the conditions (\ref{eq:ListOfCases}) hold and when $|g|<\infty$, then we have $g\neq h^2$ and $h\neq g^2$.
\end{thm}

For the case $\{l,k\}=\{3,1\}$ we introduce the following conditions from~\cite{BBP}:

\begin{description}
  \item[(BBP-E4)] $\{|g|,|h|\}=\{2,4\}$ and gp$\{g,h\}\cong \Z_2\oplus \Z_4$;
  \item[(BBP-E5)] $\{|g|,|h|\}=\{2,5\}$ and gp$\{g,h\}\cong \Z_2\oplus \Z_5$.
\end{description}

\begin{thm}[{{\cite[Theorem~4]{BBP}}}]\label{thm:lk=31}
Let $\{l,k\}=\{3,1\}$ and suppose that none of the conditions (\ref{eq:ListOfCases}) hold. Suppose further that neither (BBP-E4) nor (BBP-E5) hold. Then $\mathcal{Q}$ is weakly aspherical and hence aspherical.
\end{thm}

We now turn to the case $\{l,k\}=\{4,1\}$.

\begin{thm}[{{\cite[Theorem~2]{HM}}}]\label{thm:lk=41}
Let $\{l,k\}=\{4,1\}$ and suppose that none of the conditions (\ref{eq:ListOfCases}) hold. Then $\mathcal{Q}$ is diagrammatically reducible and hence aspherical.
\end{thm}

Generalizing Theorems~\ref{thm:lk=31} and~\ref{thm:lk=41} we have the following for the case $\{n,1\}$:

\begin{thm}[{{\cite[Theorem~1.1]{AE14}}}]\label{thm:AElk=1m}
Let $\{l,k\}=\{n,1\}$ where $n\geq 3$ and suppose that none of the conditions (\ref{eq:ListOfCases}) hold. Then $\mathcal{Q}$ is diagrammatically reducible and hence aspherical.
\end{thm}

Note that Theorem~\ref{thm:AElk=1m} shows that $\mathcal{Q}$ is diagrammatically reducible and aspherical in the previously unresolved cases (BBP-E4),(BBP-E5) of Theorem~\ref{thm:lk=31}.

We now consider cases $\{l,k\}$ with $l>0,k>0$ that are not (necessarily) of the form $\{n,1\}$. For the case $\{l,k\}=\{3,2\}$ we introduce the following condition, (which is extracted from the statement of~\cite[Theorem~1]{HM}).\footnote{There are typos in the statements of Lemma 9 and of parts (3) and (4) of Theorem 1 of \cite{HM}. The proof of \cite[Lemma 9]{HM} employs coset enumeration to check that the groups defined by $\pres{h,x}{h^k, x^3h^2x^2h}$ are finite for $2 \leq k \leq 5$. As in \cite[Question 2]{HM}, the asphericity status of $\pres{\sgp{g},x}{x^3gx^2g^2}$ is unresolved if $6<|g|<\infty$.}

\begin{description}
  \item[(HM-E)] ($g=h^2$, $6<|h|<\infty$) or ($h=g^2$, $6<|g|<\infty$).
\end{description}

\begin{thm}[{{\cite[Theorem~1]{HM}}}]\label{thm:lk=32}
Let $\{l,k\}=\{3,2\}$ and suppose that none of the conditions (\ref{eq:ListOfCases}) hold. Suppose further that (HM-E) does not hold. Then $\mathcal{Q}$ is diagrammatically reducible and hence aspherical.
\end{thm}

To consider the general case $l>0,k>0$, $l\neq k$ we introduce the following condition from~\cite{AEJ17}:
\begin{description}
  \item[(AEJ-E)]
  \begin{itemize}
    \item[(i)] $g=h^2$, $6<|h|<\infty$, $l<k<2l$; or
    \item[(ii)] $h=g^2$, $6<|g|<\infty$, $k<l<2k$; or
    \item[(iii)] $h=g^2$, $6<|g|<\infty$, $l<k<2l$; or
    \item[(iv)] $g=h^2$, $6<|h|<\infty$, $k<l<2k$.
  \end{itemize}
\end{description}

\begin{thm}[{{\cite[Theorem~1.1]{AEJ17}}}]\label{thm:AEJ17lk>0}
Let $l>0,k>0$, $l\neq k$, $\{l,k\}\neq\{2,1\}$ and suppose none of the conditions (\ref{eq:ListOfCases}) hold. Suppose further that (AEJ-E) does not hold. Then $\mathcal{Q}$ is diagrammatically reducible and hence aspherical.
\end{thm}

We do not expect there to be any non-aspherical presentations in the Case (AEJ-E):

\begin{conj}\label{conj:lk>0}
Let $l>0,k>0$, $l\neq k$. If (AEJ-E) holds then $\mathcal{Q}$ is diagrammatically reducible and hence aspherical.
\end{conj}

We now turn to the cases $l>0,k=-1$, where a less complete picture is known. We start with the case $l=2,k=-1$. For this we need to introduce the following conditions:

\begin{description}
  \item[(E-E1)] ($|g|=9$, $|h|=3$, $h=g^{3}$) or ($|h|=9$, $|g|=3$, $g=h^{3}$);
  \item[(E-E2)] ($|g|=9$, $|h|=3$, $h=g^{- 3}$) or ($|h|=9$, $|g|=3$, $g=h^{- 3}$);
  \item[(E-E3)] ($|g|=8$, $|h|=4$, $h=g^2$) or ($|h|=8$, $|g|=4$, $g=h^2$).
\end{description}

\begin{thm}[{{\cite[Theorem~1.1]{ECap}}}]\label{thm:l=2k=-1}
Suppose $l=2,k=-1$ and that none of the conditions (\ref{eq:ListOfCases}) hold and that none of (E-E1),(E-E2),(E-E3) hold. Then $\mathcal{Q}$ is diagrammatically reducible if and only if none of the following holds:
\begin{itemize}
  \item[(i)] $|g|<\infty$ and either $g=h^{-2}$ or $h=g^{-2}$;
  \item[(ii)]  $[g,h]=1$ and either $|g|=2$ or $|h|=2$;
  \item[(iii)] $\{|g|,|h|\}=\{2,3\}$ and $(gh)^2(g^{-1}h^{-1})^2=1$;
  \item[(iv)] $|g|=|h|=3$ and $[g,h]=1$;
  \item[(v)] $|g|=|h|=7$ and either $g=h^2$ or $h=g^2$;
  \item[(vi)] $|g|=|h|=9$ and either $g=h^2$ or $h=g^2$.
\end{itemize}
\end{thm}

The following corollary shows that the non-diagrammatically reducible presentations identified in Theorem~\ref{thm:l=2k=-1}, and also those in Case~(E-E3), are non-aspherical.

\begin{co}\label{cor:l=2k=-1}
Suppose $l=2,k=-1$ and that neither of (E-E1) nor (E-E2) holds.  Then $\mathcal{Q}$ is aspherical if and only if none of the conditions (\ref{eq:ListOfCases}) hold, none of Cases~(i)-(vi) of Theorem~\ref{thm:l=2k=-1} hold, and (E-E3) does not hold.
\end{co}

\begin{proof}
It suffices to show that $\mathcal{Q}$ is non-aspherical in Cases (i)--(vi) and (E-E3). By Lemma~\ref{Lemma:Coefficient} we may assume that $G=\gpres{g,h}$.

In Case~(i) we may assume that $h=g^{-2}$ and $|h|=n$. If $n\leq 5$ or $n=7$ then $G(\mathcal{Q})$ is finite of order greater than $n$, so $\mathcal{Q}$ is non-aspherical by Theorem~\ref{Theorem:Theory}(c); if $n\geq 6$ is even then $G(\mathcal{Q})$ is non-aspherical by Example~\ref{ex:3mfdgps}(a), and if $n\geq 9$ is odd then $G(\mathcal{Q})$ is non-aspherical by Example~\ref{ex:infinitewithtorsion}(a).

In Cases~(iii),(iv),(E-E3) the presentation $\mathcal{Q}$ is non-aspherical by Example~\ref{ex:333}, and in Cases~(v),(vi) it is non-aspherical by Example~\ref{ex:infinitewithtorsion}. In Case~(ii), taking $h^2=1$, then (as in~\cite{ECap}) the relation $(xgh)^2=g^{-1}(xgh)g$ can be shown to hold in $G(\mathcal{Q})$, from which it follows that $xgh$ is an element of order $2^{|g|}-1$ in $G(\mathcal{Q})$, and so $\mathcal{Q}$ is non-aspherical by Theorem~\ref{Theorem:Theory}(c).
\end{proof}

We now turn to the case $l=3,k=-1$. Theorems~1.1 and~1.2 of~\cite{AAE14} consider asphericity of one-relator relative presentations $\pres{G,x}{xg_1xg_2xg_3x^{-1}g_4}$ (see also~\cite{Ahmad}). By putting $g_1=g_2=1$, $g=g_3$, $h=g_4$, this becomes the presentation $\mathcal{Q}$ with $\{l,k\}=\{3,-1\}$. With these restrictions, the exceptional Cases (E) and (E4) of~\cite{AAE14} paper become:
\begin{description}
  \item[(AAE-E)] $\gpres{g,h}\cong\Z_2\oplus \Z_4$;
  \item[(AAE-E4)] ($|h|=8$, $g=h^4$) or ($|g|=8$, $h=g^4$).
\end{description}

As a corollary of~\cite[Theorems~1.1,1.2]{AAE14} we then have:

\begin{thm}[{{\cite[Theorems~1.1,1.2]{AAE14}}}]\label{thm:l=3k=-1}
Suppose $l=3,k=-1$ and that none of the conditions (\ref{eq:ListOfCases}) hold. Suppose further that (AAE-E) and (AAE-E4) do not hold. Then $\mathcal{Q}$ is diagrammatically reducible and hence aspherical.
\end{thm}

We now turn to the case $\{l,k\}=\{n,-1\}$ where $n\geq 4$. This is considered in~\cite{Davidson09} under the hypothesis that $g^2\neq 1,h^2\neq 1$ and the exclusion of three exceptional families:

\begin{description}
\item[(D-E1)] ($g=h^2$, $3<|h|<\infty$) or ($h=g^2$, $3<|g|<\infty$);
\item[(D-E2)] ($g=h^{-2}$, $3<|h|<\infty$) or ($h=g^{-2}$, $3<|g|<\infty$);
\item[(D-E4)] ($g=h^3$, $|h|=9$) or ($h=g^3$, $|g|=9$).
\end{description}

\begin{thm}[{\cite{Davidson09}}]\label{thm:l>3k=-1}
Suppose $l\geq 4$, $k=-1$, that $g^2\neq 1, h^2\neq 1$ in $G$, and that none of ($Z$), (${M}$), ($P$) hold. Suppose further that none of (D-E1),(D-E2),(D-E4) hold. Then $\mathcal{Q}$ is diagrammatically reducible and hence aspherical.
\end{thm}

Note that the hypotheses of Theorem~\ref{thm:l>3k=-1} imply that none of the cases given in (\ref{eq:ListOfCases}) hold. The more general case $l>0,k<0$ was considered in~\cite{P95} under stronger hypotheses on relations involving the elements $g,h$ of $G$. The following theorem, as stated in~\cite{P95}, includes the hypothesis that the natural homomorphism $G\ra G(\mathcal{Q})$ is injective. That hypothesis is redundant however, as (under the remaining hypotheses) injectivity follows from~\cite{Edj91b}.

\begin{thm}[{{\cite[Theorem~C(1)]{P95}}}]\label{thm:Prishchepov}
Suppose that $l>0,k<0$ and $l>2|k|$ or $|k|>2l$, that Cases~($Z$),(${M}$) do not hold, and that one of the following holds:
\begin{itemize}
  \item[(i)] $|g|\geq 6$, $|h|\geq 3$, $h\neq g^{\pm 2}$, $h\neq g^{-3}$, $g\neq h^{-2}$; or
  \item[(ii)] $|g|\geq 3$, $|h|\geq 6$, $g\neq h^{\pm 2}$, $g\neq h^{-3}$, $h\neq g^{-2}$; or
  \item[(iii)] $|g|\geq 4$, $|h|\geq 4$, $g\neq h^{-2}$, $h\neq g^{-2}$.
\end{itemize}
Then $\mathcal{Q}$ is aspherical.
\end{thm}

We close with the following asphericity result, which was used to complete the classification of the finite Fibonacci groups~$F(r,n)$. Its proof highlights the intricacy of the arguments that can be required when proving asphericity of presentations $\mathcal{Q}$ when~(\ref{eq:shortcycles}) holds.

\begin{thm}[{{\cite[Theorem~1.2]{EdjJuh14}}}]\label{thm:K_5(l,-1)}
Let $l\geq 7$, $k=-1$ and suppose that ($K_5$) holds. Then $\mathcal{Q}$ is diagrammatically reducible and hence aspherical.
\end{thm}

\end{document}